\documentclass[a4paper,10pt,twoside]{article}
\usepackage{pst-all}
\usepackage{graphicx}
\usepackage{fancyhdr}
\usepackage{amsmath}
\usepackage{epsfig} 
\usepackage{psfrag} 
\usepackage{amsmath,amsfonts,amssymb,amsthm} 
\usepackage{bbm} 

\newtheorem{Theorem}{Theorem}[section] 
\newtheorem{Corollary}[Theorem]{Corollary} 
\newtheorem{Lemma}[Theorem]{Lemma}

\renewcommand{\proofname}{Proof} 
\theoremstyle{definition} 

\newcommand{\CCPPst}{P_{s,t-\mbox{\scriptsize path}}^{c}(D)}
\newcommand{\CCPPstl}{P_{s,t-\mbox{\scriptsize path}}^{(c_1)}(D)}
\newcommand{\CCPPstu}{P_{s,t-\mbox{\scriptsize path}}^{(c_m)}(D)}
\newcommand{\CCPPon}{P_{0,n-\mbox{\scriptsize path}}^c(\tilde{D}_n)}
\newcommand{\tdn}{\tilde{D}_n}
\newcommand{\tnn}{\tilde{N}_n}
\newcommand{\tann}{\tilde{A}_n}
\newcommand{\OddPath}{P_{0,n-\mbox{\scriptsize path}}^{\mbox{odd}}(D)}
\newcommand{\EvenPath}{P_{0,n-\mbox{\scriptsize path}}^{\mbox{even}}(D)}
\newcommand{\CCPPonl}{P_{0,n-\mbox{\scriptsize path}}^{(c_1)}(D)}
\newcommand{\CCPPkon}{P_{0,n-\mbox{\scriptsize path}}^{(k)}(D)}
\newcommand{\CCCP}{P_C^{c}(D)}
\newcommand{\CCCPn}{P_C^{c}(D_n)}
\newcommand{\uCCPPst}{P_{s,t-\mbox{\scriptsize path}}^{c}(G)}
\newcommand{\uCCPPon}{P_{0,n-\mbox{\scriptsize path}}^{c}(K_{n+1})}
\newcommand{\uCCCP}{P_C^{c}(G)}
\newcommand{\uCCCPn}{P_C^{c}(K_n)}
\newcommand{\CHS}{\mbox{CHS}}
\newcommand{\CS}{\mbox{CS}}
\newcommand{\symmetric}{pseudo-symmetric }

\title{On cardinality constrained cycle and path polytopes}
\author{Volker Kaibel and R\"udiger Stephan}
\date{}
\begin{document}

\maketitle

\begin{abstract}
Given a directed graph $D=(N,A)$ and a sequence of positive integers
$1 \leq c_1< c_2< \dots < c_m \leq |N|$, we consider those path and
cycle polytopes that are defined as the convex hulls of simple paths
and cycles of $D$ of cardinality $c_p$ for some $p \in \{1,\dots,m\}$,
respectively. We present integer characterizations of these polytopes
by facet defining linear inequalities for which the separation problem can be solved in polynomial time.
These inequalities can simply be transformed into
inequalities that characterize the integer points of the
undirected counterparts of cardinality constrained path and cycle
polytopes. Beyond we investigate some further inequalities, in
particular inequalities that are specific to odd/even paths and cycles.
\end{abstract}

\section{Introduction}
Let $D=(N,A)$ be a directed graph on $n$ nodes that has
neither loops nor parallel arcs, and let $c =
(c_1,\dots,c_m)$ be a nonempty sequence of integers such that $1 \leq
c_1 < c_2 < \dots < c_m \leq n$ holds. Such a sequence is called a
\emph{cardinality sequence}. 
For two different nodes $s,t \in N$, the \emph{cardinality constrained
  (s,t)-path polytope}, denoted by $\CCPPst$, is the convex hull of the
incidence vectors of simple directed $(s,t)$-paths $P$ such that $|P| = c_p$
holds for some $p \in \{1,\dots,m\}$.
The \emph{cardinality constrained cycle polytope} $\CCCP$, similar
defined, is the convex hull of the incidence vectors of simple
directed cycles $C$ with $|C|=c_p$ for some $p$. Note, since $D$ does not have loops, we may assume
$c_1 \geq 2$ when we investigate cycle polytopes. 
The undirected counterparts of these
polytopes  are defined similarly. We denote them by $\uCCPPst$ and
$\uCCCP$, where $G$ is an undirected graph. The associated polytopes
without cardinality restrictions we denote by $P_{s,t-\mbox{\scriptsize
    path}}(D)$, $P_{s,t-\mbox{\scriptsize path}}(G)$, $P_C(D)$, and $P_C(G)$.

Cycle and path polytopes, with and without cardinality restrictions,
defined on graphs or digraphs, are
already well studied. For a literature survey on these polytopes see
Table 1.

\begin{table}
\caption{\bf Literature survey on path and cycle polyhedra}

\vspace{0.4cm} \noindent
\begin{tabular}{|rl|} \hline 
&\\
Schrijver \cite{Schrijver2003}, chapter 13: & dominant of $P_{s,t-\mbox{\scriptsize path}}(D)$\\
Stephan \cite{Stephan}: & $P_{s,t-\mbox{\scriptsize path}}^{(k)}(D)$\\
Dahl, Gouveia \cite{DG}: & $P_{s,t-\mbox{\scriptsize path}}^{\leq
  k}(D):= P_{s,t-\mbox{\scriptsize path}}^{(1,\dots,k)}(D)$\\
 Dahl, Realfsen \cite{DR}: & $P_{s,t-\mbox{\scriptsize path}}^{\leq
  k}(D)$, $D$ acyclic\\ 
Nguyen \cite{Nguyen}: & dominant of  $P_{s,t-\mbox{\scriptsize path}}^{\leq
  k}(G)$\\
Balas, Oosten \cite{BO}: & directed cycle polytope $P_C(D)$\\
Balas, Stephan \cite{BST}: & dominant of $P_C(D)$\\
Coullard, Pulleyblank \cite{CP}, Bauer \cite{Bauer}: & undirected
cycle polytope $P_C(G)$\\ 
Hartmann, \"Ozl\"uk \cite{HO}: & $P_C^{(k)}(D)$\\
Maurras, Nguyen \cite{MN1,MN2}: &  $P_C^{(k)}(G)$\\
Bauer, Savelsbergh, Linderoth \cite{BLS}: &  $P_C^{\leq k}(G)$\\ [3mm]\hline
\end{tabular}
\end{table}

Those publications that treat cardinality
restrictions, discuss only the cases $\leq k$ or $= k$, while we
address the general case. In particular, we assume $m \geq 2$.
The main contribution of this paper will be the presentation of
IP-models (or IP-formulations) for cardinality constrained path and
cycle polytopes whose inequalities generally define facets with
respect to complete graphs and digraphs. Moreover, the associated separation problem can be solved in polynomial time.

The basic idea of this paper can be presented best for cycle
polytopes. Given a finite set $B$ and a cardinality sequence $b=(b_1,\dots,b_m)$, the set $\CHS^{b}(B):=\{F \subseteq B : |F|=b_p \mbox{ for some
} p\}$ is called a \emph{cardinality homogenous set system}. 
Clearly, $P_C^c(D) = \mbox{conv} \{\chi^C \in \mathbb{R}^A\;|\;
C \mbox{ simple cycle}, \, C \in CHS^c(A)\}$, where $CHS^{c}(A)$ is the cardinality homogeneous
set system defined on the arc set $A$ of $D$. According to Balas and Oosten
\cite{BO}, the integer points of the cycle polytope $P_C(D)$ can be
characterized by the system
\begin{equation} \label{model1}
\begin{array}{rcll}
x(\delta^{\mbox{\scriptsize out}}(i))- x(\delta^{\mbox{\scriptsize in}}(i)) & = & 0   & \mbox{for all } i \in N,\\ 
x(\delta^{\mbox{\scriptsize out}}(i)) & \leq & 1 & \mbox{for all } i \in N,\\
 -  x((S:N \setminus S)) + x(\delta^{\mbox{\scriptsize out}}(i))+ x(\delta^{\mbox{\scriptsize out}}(j)) & \leq & 1&
\mbox{for all }
S \subset N,\\
& & &  2 \leq |S| \leq n-2,\\
& & &  i \in S, j \in N \setminus S,\\
x(A) & \geq & 2,\\
\multicolumn{3}{r}{x_{ij} \in \{0,1\}}  & \mbox{for all } (i,j) \in A.
\end{array}
\end{equation}
Here, $\delta^{\mbox{\scriptsize out}}(i)$ and $\delta^{\mbox{\scriptsize in}}(i)$ denote the set of arcs leaving
and entering node $i$, respectively; for an arc set $F \subseteq A$
we set $x(F):=\sum_{(i,j) \in F} x_{ij}$; for any subsets $S,T$ of 
$N$, $(S:T)$ denotes $\{(i,j) \in A| i \in S, j \in T\}$.
Moreover, for any $S \subseteq N$, we denote by $A(S)$ the subset of arcs whose
both endnodes are in $S$.

Gr\"otschel \cite{Groetschel} presented a complete linear description of
a cardinality homogeneous set system. For $CHS^{c}(A)$, the model reads:
\begin{equation} \label{model2}
\begin{array}{l}
\hspace{3.7cm}
\begin{array}{rcccll}
0   & \leq &  x_{ij}& \leq & 1 & \mbox{for all } (i,j) \in A,\\
c_1 & \leq &  x(A)  & \leq & c_m,
\end{array} \\
\\
(c_{p+1} - |F|) \; x(F)-(|F| - c_p) \; x(A \setminus F) \leq 
c_p(c_{p+1}-|F|) \\
\hspace{1cm} \mbox{for all } F \subseteq A \mbox{ with } c_p < |F| <
c_{p+1} \mbox{ for some } p \in \{1,\dots,m-1\}.
\end{array}
\end{equation}
The \emph{cardinality bounds} $c_1 \leq  x(A) \leq c_m$ exclude
all subsets of $A$ whose cardinalities are out of the bounds
$c_1$ and $c_m$, while the latter class of inequalities of model
(\ref{model2}), which are called \emph{cardinality forcing
  inequalities}, cut off all 
arc sets $F \subseteq A$ of forbidden cardinality between the bounds, since for each
such $F$, the cardinality forcing inequality associated with $F$ is
violated by $\chi^F$:
\[
(c_{p+1}-|F|) \chi^F(F)-(|F|-c_p)\chi^F(A \setminus F)=
|F|(c_{p+1}-|F|)> c_p(c_{p+1}-|F|).
\] 
However, for any $H \in CHS^{c}(A)$ the inequality associated with $F$ is valid. If $|H| \leq c_p$,  
then $(c_{p+1}-|F|) \chi^H(F)-(|F|-c_p)\chi^H(A \setminus F) \leq
(c_{p+1}-|F|) x(H \cap F) \leq c_p (c_{p+1}-|F|)$, and equality holds if $|H|=c_p$ and $H \subseteq F$.
If $|H| \geq c_{p+1}$, then  $(c_{p+1}-|F|) \chi^H(F)-(|F|-c_p)\chi^H(A \setminus F) \leq |F|(c_{p+1}-|F|) -  (c_{p+1}-|F|)(|F|-c_p) = c_p (c_{p+1}-|F|)$, and equality holds if $|H|=c_{p+1}$ and $H \cap F = F$.

Combining both models results obviously in an integer characterization for the cardinality
constrained cycle polytope $\CCCP$. However, the cardinality forcing
inequalities in this form are quite weak, that is, they define very
low dimensional faces of $\CCCP$. 
The key for obtaining stronger cardinality forcing inequalities for $P_C^c(D)$ is to count the nodes of a cycle rather than its arcs. The trivial, but crucial observation here is that, for the incidence vector $x\in\{0,1\}^A$ of a cycle in~$D$ and for every node~$i\in V$, we have $x(\delta^{\mbox{\scriptsize out}}(i))=1$ if  the cycle contains node~$i$, and $x(\delta^{\mbox{ \scriptsize out}}(i))=0$ if it does not. Thus, for every $W\subseteq N$ with $c_p<|W|<c_{p+1}$ for some $p\in\{1,\dots,m-1\}$,
the  cardinality-forcing inequality 
\begin{equation*}
   (c_{p+1} - |W|) \sum_{i \in W} x(\delta^{\mbox{\scriptsize out}}(i)) 
   - (|W| - c_p) \sum_{i \in N \setminus W} x(\delta^{\mbox{\scriptsize out}}(i)) 
   \leq c_p(c_{p+1} - |W|),
\end{equation*}
is valid for $\CCCP$, cuts
off all cycles $C$, with $c_p < |C| < c_{p+1}$, that visit $\min
\{|C|, |W|\}$ nodes of $W$, 
and is satisfied with equation by all cycles of
cardinality $c_p$ or $c_{p+1}$ that visit $\min
\{|C|, |W|\}$  nodes of $W$. 
Using these inequalities yields the following integer characterization for $\CCCP$:
\begin{equation} \label{model3}
\begin{array}{rcll}
x(\delta^{\mbox{\scriptsize out}}(i))- x(\delta^{\mbox{\scriptsize in}}(i)) & = & 0   & \mbox{for all } i \in N,\\ 
x(\delta^{\mbox{\scriptsize out}}(i)) & \leq & 1 & \mbox{for all } i \in N,\\
 -  x((S:N \setminus S)) + x(\delta^{\mbox{\scriptsize out}}(i))+ x(\delta^{\mbox{\scriptsize out}}(j)) & \leq & 1&
\mbox{for all }
S \subset N,\\
& & &  2 \leq |S| \leq n-2,\\
& & &  i \in S, j \in N \setminus S,\\ \\
x(A) & \geq & c_1,\\
x(A) & \leq & c_m,\\
\\
(c_{p+1} - |W|) \sum_{i \in W} x(\delta^{\mbox{\scriptsize out}}(i)) & & & \\
- (|W|- c_p) \sum_{i \in N \setminus W} x(\delta^{\mbox{\scriptsize out}}(i)) & & & \\
- c_p (c_{p+1} - |W|) & \leq & 0 & \forall \; W \subseteq N: \; \exists p\\ 
& & & \mbox{with } c_p < |W| <c_{p+1},\\ \\
x_{ij}&  \in&  \{0,1\}  & \mbox{for all } (i,j) \in A.
\end{array}
\end{equation}

However, in the polyhedral analysis of cardinality constrained path and
cycle polytopes we will focus on the directed cardinality
constrained path polytope for
a simple reason: valid inequalities
for $\CCPPst$ can easily be transformed into valid
inequalities for the other polytopes. In particular,
from the IP-model for $\CCPPst$ that we present in section 3 we derive IP-models for the remaining polytopes $\mathcal{P}$,
as illustrated in Figure 1, such that a transformed inequality is facet
defining for $\mathcal{P}$ when the original inequality is facet
defining for $\CCPPst$. In addition, the subpolytopes
$P_{s,t-\mbox{\scriptsize path}}^{(c_p)}(D)$  of
$P_{s,t-\mbox{\scriptsize path}}^{c}(D)$ were studied in \cite{Stephan}. Theorem \ref{T3} in Section 2 and
Table 1 in \cite{Stephan} imply that they are of codimension 1
whenever $4 \leq c_p \leq n-1$,
provided that we have an appropriate digraph $D$. Thus, any facet
defining inequality $\alpha x \leq \alpha_0$ for $P_{s,t-\mbox{\scriptsize path}}^{(c_p)}(D)$
which is also valid for $P_{s,t-\mbox{\scriptsize path}}^{c}(D)$ can
easily be shown to be facet defining also for
$P_{s,t-\mbox{\scriptsize path}}^{c}(D)$ if $\alpha y =\alpha_0$ holds
for some $y \in P_{s,t-\mbox{\scriptsize path}}^{c}(D) \setminus
P_{s,t-\mbox{\scriptsize path}}^{(c_p)}(D)$. So, in the present 
paper many facet proofs must not be given from the
scratch, but can be traced back to results in \cite{Stephan}. 

\vspace{0.5cm}
\begin{center}
\psset{xunit=1cm,yunit=1cm,runit=1cm} 
\begin{pspicture}(-4,-6)(6,2) 

\rput(-3,1.5){\rnode{1}{ \psframebox[linecolor=black]{$\uCCCP$}}} 
\rput(3,1.5){\rnode{2}{ \psframebox[linecolor=black]{$\CCCP$}}} 
\rput(-3,-1.5){\rnode{3}{ \psframebox[linecolor=black]{$\uCCPPst$}}} 
\rput(3,-1.5){\rnode{4}{ \psframebox*[fillcolor=yellow,
    linecolor=black]{$\CCPPst$}}}
 \rput(3,-1.5){\rnode{8}{ \psframebox[fillcolor=yellow, linecolor=black]{$\CCPPst$}}}
\rput(3,-4.5){\rnode{5}{ \psframebox[linecolor=black]{ $P_{s,t-\mbox{\scriptsize path}}^{(c_p)}(D)$}}} 

\ncline{->}{2}{1}\nbput{deorienting}
\ncline{->}{4}{2}\nbput{lifting}
\ncline{->}{4}{3}\nbput{deorienting}
\ncline{->}{5}{4}\nbput{lifting}

\rput(0,-5.5){Figure 1. $\CCPPst$ and related polytopes.}
\end{pspicture}
\end{center}

In the following we investigate the cardinality constrained path polytope
$ P_{0,n-\mbox{\scriptsize path}}^c(D)$ defined on a digraph $D=(N,A)$ with node set $N =
\{0,\dots,n\}$. In particular, $s=0, t=n$. Since $(0,n)$-paths do not use arcs entering $0$ or leaving $n$, we may assume that $\delta^{\mbox{\scriptsize in}}(0)= \delta^{\mbox{\scriptsize out}}(n)=
\emptyset$. Next, suppose that $A$ contains the arc $(0,n)$ and the
cardinality sequence $c$ starts with $c_1=1$. Then the equation
\[
\dim P_{0,n-\mbox{\scriptsize path}}^{(c_1,c_2,\dots,c_m)}(D) = \dim
P_{0,n-\mbox{\scriptsize path}}^{(c_2,\dots,c_m)}(D)+1
\] 
obviously holds.
Moreover, an inequality $\alpha x \leq \alpha_0$ defines a facet of
$P_{0,n-\mbox{\scriptsize path}}^{(c_2,\dots,c_m)}(D)$ if and only if
the inequality $\alpha x + \alpha_0 x_{0n} \leq \alpha_0$ defines a
facet of $P_{0,n-\mbox{\scriptsize
    path}}^{(1,c_2,\dots,c_m)}(D)$. Thus, the consideration of
cardinality sequences starting with 1 does not give any new insights
into the facial structure of cardinality constrained path
polytopes. So we may assume that $A$ does not contain the arc $(0,n)$.
So, for our purposes it suffices to suppose that the arc set $A$ of $D$
is given by 
\begin{equation} \label{arcset}
A=\{(0,i), (i,n) : i=1,\dots,n-1\} \bigcup \{(i,j) : 1 \leq i,j \leq n-1, i \neq
j\}.
\end{equation} 
Therefore, by default, we will deal with the directed graph $\tilde{D}_n = (\tilde{N}_n, \tilde{A}_n)$, where $\tilde{N}_n=\{0,1,\dots,n\}$ and $\tilde{A}_n=A$ is \eqref{arcset}. 

The remainder of the paper is organized as follows: In Section 2, we examine
the relationship between directed path and cycle
polytopes. In Section 3, we consider the inequalities of the IP-model
for the directed cardinality constrained path polytope $\CCPPon$  and
give necessary and sufficient conditions for them to be facet
defining. Moreover, we present some
further classes of inequalities that also cut off forbidden
cardinalities. Finally, in Section 4,
we transform facet defining inequalities for $\CCPPon$ into facet
defining inequalities for the other polytopes. 

\section{The relationship between directed path and cycle polytopes} 

This section generalizes the results in \cite{Stephan}, Section 2.
Denote by $\cal{P}$ the set of simple
$(0,n)$-paths $P$ in $\tdn=(\tnn,\tann)$. Let $D'$ be the digraph that arises by
removing node $0$ from $\tdn$ and identifying $\delta^{\mbox{\scriptsize out}}(0)$ with
$\delta^{\mbox{\scriptsize out}}(n)$. Then, $D'$ is a complete digraph on node set
$\{1,\dots,n\}$ and $\cal{P}$ becomes the set $\mathcal{C}^n$ of
simple cycles that visit node $n$.  
The convex hull of the incidence vectors of cycles $C \in \mathcal{C}^n$ 
in turn is the restriction of the cycle
polytope defined on $D'$ to the hyperplane $x(\delta^{\mbox{\scriptsize out}}(n))=1$. Balas
and Oosten \cite{BO} showed that the \emph{degree constraint} 
\[x(\delta^{\mbox{\scriptsize out}}(i)) \leq 1\] 
induces a facet of the cycle polytope defined on a complete
digraph. Hence, the path polytope
$P_{0,n-\mbox{\scriptsize path}}(\tdn)$ is isomorphic to a facet of the
cycle polytope $P_C(D')$. From the next theorem we conclude that
this relation holds also for cardinality constrained path and cycle
polytopes. We start with some preliminary statements from linear algebra.

\begin{Lemma} \label{L1}
Let $k \neq \ell$ be natural numbers, let $x^1,x^2,\dots,x^r \in \mathbb{R}^p$
be vectors satisfying the equation $1^T x^i = k$, and let $y \in
\mathbb{R}^p$ be a vector satisfying the equation $1^Ty=\ell$. 
Then the following holds:\\[1em]
(i) $y$ is not in the affine hull of the set $\{x^1,\dots,x^r\}$.\\[0.5em]
(ii) The points $x^1,\dots,x^r$ are affinely independent if and only if they are linearly independent.
\hfill $\Box$
\end{Lemma}

According to the terminology of Balas and Oosten~\cite{BO}, for any digraph $D=(N,A)$ on $n$ nodes we call the
polytope 
$$P_{CL}^c(D):= \{(x,y) \in P_C^c(D) \times \mathbb{R}^n :
y_i=1-x(\delta^{\mbox{\scriptsize out}}(i)),i=1,\dots,n\}$$
the \emph{cardinality
  constrained cycle-and-loops polytope}.  
Its integer points are the incidence vectors of spanning unions of a
simple cycle and loops.  

\begin{Lemma} \label{L2}
The points $x^1,\dots,x^p \in \CCCP$ are affinely independent if and
only if the corresponding points $(x^1,y^1),\dots,(x^p,y^p) \in
P_{CL}^c(D)$ are affinely independent. 
\end{Lemma}

\begin{proof} 
The map $f : P_{CL}^c(D) \to P_C(D), \; (x,y) \mapsto x$ is an affine isomorphism.
\end{proof}

\begin{Theorem} \label{T3}
Let $D_n=(N,A)$ be the complete digraph on $n \geq 3$ nodes and
$c=(c_1,\dots,c_m)$ a cardinality sequence with $m \geq 2$. Then the following
holds:\\[1em]
(i) The dimension of $\CCCPn$ is $(n-1)^2$.\\[0.5em]
(ii) For any node $i \in N$, the degree inequality $x(\delta^{\mbox{\scriptsize out}}(i))
  \leq 1$ defines a facet of $\CCCPn$.
\end{Theorem}

\begin{proof}
(i) Balas and Oosten \cite{BO} proved that $\dim
P_C(D_n)=(n-1)^2$, while Theorem 1 of Hartmann and \"Ozl\"uk \cite{HO}
says that
\begin{equation} \label{dimeq}
\dim P_C^{(k)}(D_n)= 
\left \{ 
\begin{array}{ll}
|A|/2 -1,& \mbox{if } k=2,\\
n^2-2n, & \mbox{if } 2<k<n \mbox{ and } n \geq 5,\\
n^2-3n+1, & \mbox{if } k=n \mbox{ and } n \geq 3,
\end{array}
\right .
\end{equation}
and $\dim P_C^{(3)}(D_4)= 6$.
Since $\CCCPn \subseteq P_C(D_n)$, it follows immediately that $\dim
\CCCPn \leq (n-1)^2$. When $n=3$, $m \geq 2$ implies
$\CCCPn=P_C(D_n)$, and thus $\dim P_C^{(2,3)}(D_3)=4$. When $n=4$, the
statement can be verified using a computer program, for instance, with
\texttt{polymake} \cite{polymake}. For $n \geq 5$ the claim follows from
\eqref{dimeq} and Lemma \ref{L1} (i) unless
$c=(2,n)$: it exists some cardinality $c_p$, with
$2<c_p<n$, and thus there are $n^2-2n+1$ affinely independent vectors
$x^r \in P_C^{(c_p)}(D_n) \subset \CCCPn$. Moreover, since $m \geq 2$,
there is a vector $y \in \CCCPn$ of another cardinality which is
affinely independent from the points $x^r$. Hence,
$\CCCPn$ contains $n^2-2n+2$ affinely independent points proving
$\dim \CCCPn = (n-1)^2$. 

When $c=(2,n)$, the above argumentation fails, since
the dimensions of both polytopes 
$P_C^2(D_n)$ and $P_C^n(D_n)$ are less than $n^2-2n$. Setting
$d_n:=\dim P_C^{(n)}(D_n)$, we see that there are $d_n+1=n^2-3n+2$ linearly
independent points $x^r \in P_C^{(2,n)}(D_n) \cap
P_C^{(n)}(D_n)$ satisfying $1^T x^r = n$. Clearly, the points $(x^r,y^r) \in P_{CL}^{(2,n)}$
are also linearly independent.
Next, consider the point 
$(x^{23},y^{23})$, where $x^{23}$
 is the incidence vector of the 2-cycle $\{(2,3),(3,2)\}$,
and $n-1$ further points 
$(x^{1i},y^{1i})$, where $x^{1i}$ is the incidence vector of the 2-cycle $\{(1,i),(i,1)\}$. 
The incidence matrix Z whose rows are the vectors
$(x^r,y^r)$, $r=1,2,\dots,d_n+1$, $(x^{23},y^{23})$, and $(x^{1i},y^{1i})$, $i=2,3,\dots,n$,
is of the form 
\[Z=
\begin{pmatrix}
X & \mathbf{0} \\
Y & L \\
\end{pmatrix}, 
\]
where 
\[L=
\left(
\begin{array}{c|c}
1 & 0 \hspace{0.2cm} 0 \hspace{0.2cm} 1  \cdots  1 \\ \hline
\\[-.5em]
\mathbf{0} & E-I \\
\end{array} \right) .
\]
E is the $(n-1) \times (n-1)$ matrix of all ones and I the $(n-1) \times (n-1)$ identity matrix. $E-I$ is nonsingular, and thus
$L$ is of rank $n$. $X$ is of rank $d_n+1$, and hence rank $(Z)=d_n+1+n=n^2-2n+2$. Together with 
Lemma \ref{L2}, this yields the desired result. 

(ii) When $n \leq 4$, the statement can be verified using a computer program. When $n \geq 5$ and $4 \leq c_p < n$ for some index $p \in \{1, \dots,m\}$,
the claim can be showed along the lines of the proof to part (i) using
Theorem 11 of Hartmann and \"Ozl\"uk \cite{HO} saying that the degree constraint
defines a facet of $P_C^{(c_p)}(D_n)$. 

It remains to show that the claim is true for $c \in \{(2,3),(2,n),(3,n),(2,3,n)\}$, $n \geq 5$. W.l.o.g. consider the inequality $x(\delta^{\mbox{\scriptsize out}}(1)) \leq 1$. When $c=(2,3)$, consider all 2- and 3-cycles whose incidence vectors satisfy $x(\delta^{\mbox{\scriptsize out}}(1))=1$. This are exactly $n^2-2n+1$ cycles, namely the 2-cycles $\{(1,j),(j,1)\}$, $j=2,\dots,n$, and the 3-cycles $\{(1,j),(j,k),(k,1)\}$ for all arcs $(j,k)$ that are not incident with node $1$. Their incidence vectors are affinely independent, and hence, the degree constraint is facet defining for $P_C^{(2,3)}(D_n)$. This implies also that it induces a facet of $P_C^{(2,3,n)}(D_n)$. Turning to the case $c=(2,n)$, note that the degree constraint is satisfied with equality by all Hamiltonian cycles. Hence, we have $d_n+1$ linearly independent Hamiltonian cycles and again, the 2-cycles $\{(1,i),(i,1)\}$, which are linearly independent of them. Finally, let $c=(3,n)$. Beside $d_n+1$ Hamiltonian cycles, consider the 3-cycles $(1,3),(3,4),(4,1)$ and $\{(1,2),(2,j),(j,1)\}$, $j=3,\dots,n$. Then the $n^2-2n+1$ corresponding points in $P_{CL}^c(D_n)$ build a nonsingular matrix. Hence, by Lemma \ref{L2}, it follows the desired result.  
\end{proof}

Given a cardinality sequence $c=(c_1,\dots,c_m)$ with $m \geq 2$ and $c_1 \geq 2$, Theorem \ref{T3} implies that $\dim \CCPPon = n^2-2n$. 
From Theorem \ref{T3} another important fact  can be derived. Facet defining inequalities for $\CCPPon$ can easily be lifted to facet defining inequalities for $\CCCPn$.
For sequential lifting, see Nemhauser and Wolsey \cite{NW}.

\begin{Theorem} \label{lifting}
Let $c=(c_1,\dots,c_m)$ be a cardinality sequence with $m \geq 2$ and $c_1 \geq 2$. Let $\alpha x \leq \alpha_0$ be a facet defining inequality for $P_{0,n-\mbox{\scriptsize path}}^{c}(\tdn)$ and $\gamma$ the maximum
of $\alpha(C)$ over all cycles $C$ in $\tdn$ with $|C|=c_p$ for some $p$. Setting $\alpha_{ni}:=\alpha_{0i}$ for $i=1,\dots,n-1$, the inequality
\begin{equation}
\sum_{i=1}^n \sum_{j=1 \atop j \neq i}^n \alpha_{ij} x_{ij} + (\gamma -\alpha_0)x(\delta^{\mbox{\scriptsize out}}(n)) \leq \gamma
\end{equation}
defines a facet of $P_C^{c}(D_n)$. \hfill $\Box$
\end{Theorem}

No similar relationship seems to hold between undirected cycle and path
polytopes.

\section{Facets of $\CCPPon$} \label{SecFacets}

Let $D=(N,A)$ be a digraph on node set $N=\{0,\dots,n\}$.
The integer points of $P_{0,n-\mbox{\scriptsize path}}^c(D)$ are characterized by the following system:
\begin{equation} \label{pathmodel}
\begin{array}{rcll}
x(\delta^{\mbox{\scriptsize out}}(i))- x(\delta^{\mbox{\scriptsize in}}(i)) & = & \multicolumn{2}{l}{
 \left \{ \begin{array}{r@{}l}
1 & \mbox{ if } i=0,\\
0 &\mbox{ if }  i \in N \setminus \{0,n\},\\
-1 & \mbox{ if }  i=n,\\
\end{array} \right.} \\
x(\delta^{\mbox{\scriptsize out}}(i)) & \leq & 1 & \mbox{for all } i \in N \setminus \{0,n\},\\
x((S:N \setminus S)) -  x(\delta^{\mbox{\scriptsize in}}(j)) & \geq & 0 & \forall  S \subset
N: 0,n \in S, j \in N \setminus S,\\ \\
x(A) & \geq & c_1,\\
x(A) & \leq & c_m,\\
\\
(c_{p+1} - |W|+1) \sum_{i \in W} x(\delta^{\mbox{\scriptsize out}}(i)) & & & \\
- (|W|-1 - c_p) \sum_{i \in N \setminus W} x(\delta^{\mbox{\scriptsize out}}(i)) & & & \\
- c_p (c_{p+1} - |W|+1) & \leq & 0 & \forall \; W \subseteq N: \; 0,n \in W, \; \exists p\\ 
& & & \mbox{with } c_p < |W|-1 <c_{p+1},\\ \\
x_{ij}&  \in&  \{0,1\}  & \mbox{for all } (i,j) \in A.
\end{array}
\end{equation}
Here, the cardinality forcing inequalities arise in another form, since the number of nodes that are visited by a simple path is one more than the number of arcs in difference to a simple cycle.
The first three and the integrality constraints ensure that $x$ is the
incidence vector of a simple $(0,n)$-path $P$ (cf. \cite{Stephan}). The
cardinality bounds and the
cardinality forcing inequalities guarantee that $|P|=c_p$ for some
$p$.

Dahl and Gouveia \cite{DG} gave a complete linear description of $P_{0,n-\mbox{\scriptsize
    path}}^{(1,2,3)}(D')$, where $D' = D \cup \{(0,n)\}$. So, we have also one for 
 $P_{0,n-\mbox{\scriptsize path}}^{(2,3)}(D)$. Consequently, from now
 on we exclude the case $c=(2,3)$ with respect to directed path
 polytopes. More precisely, in the sequel we consider only the set of
 cardinality sequences $\CS := \{c=(c_1,\dots, c_m) : m \geq 2, 2 \leq
 c_1 < \dots < c_m \leq n, c \neq (2,3) \}$.
However, as the proof of Theorem \ref{T3} indicates, the polyhedral
analysis of $\CCPPon$ becomes much harder if $c \in
\{(2,n),(3,n),(2,3,n)\}$. In order to avoid that the paper is surcharged
with long argumentations, we skip in particular these cases
and refer the interested reader to~\cite{KS}. 

Given a valid inequality $cx \leq c_0$, a $(0,n)$-path $P$ is said to be
\emph{tight} if $c(P)=c_0$.
Due to the flow conservation constraints, two different inequalities
that are valid for $P_{0,n-\mbox{\scriptsize path}}^c(D)$ may define
the same face. The next theorem, which is an adaption of a result of
Hartmann and \"Ozl\"uk \cite{HO}, says how those inequalities can be
identified. 
 
\begin{Theorem} \label{equiv}
Let $\alpha x \geq \alpha_0$ be a valid inequality for
$P_{0,n-\mbox{\scriptsize path}}^c(D)$ and let $T$ be a spanning tree
of $D$. Then for any specified set of coefficients $\beta_{ij}$ for
the arcs $(i,j) \in T$, there is an equivalent inequality $\alpha' x
\geq \alpha_0$ for $P_{0,n-\mbox{\scriptsize path}}^c(D)$ such that
$\alpha'_{ij} = \beta_{ij}$ for $(i,j) \in T$. \hfill $\Box$ 
\end{Theorem}

\subsection{Facets related to cardinality restrictions}

The cardinality bounds $x(\tann) \geq c_1$ and $x(\tann) \leq c_m$ define
facets of  the cardinality constrained path polytope $\CCPPon$ if and only if $4 \leq c_i \leq n-1$ for $i=1,m$
(see Table 1 of  \cite{Stephan}).

Next, we turn to the cardinality forcing inequalities. Due to the easier notation, we analyze them for the polytope $P^* := \{x \in \CCCPn | x(\delta^{\mbox{\scriptsize out}}(1))=1\}$
which is isomorphic to $\CCPPon$.

\begin{Theorem} \label{TCF}
  Let $D_n=(N,A)$ be the complete digraph on $n \geq 4$ nodes and $W$ a subset of $N$ with $1 \in W$ and $c_p < |W| <
  c_{p+1}$ for some $p \in \{1,\dots,m-1\}$. The cardinality-forcing inequality 
  \begin{equation} \label{CF}
   (c_{p+1} - |W|) \sum_{i \in W} x(\delta^{\mbox{\scriptsize out}}(i)) 
   - (|W| - c_p) \sum_{i \in N \setminus W} x(\delta^{\mbox{\scriptsize out}}(i)) 
   \leq c_p(c_{p+1} - |W|)
  \end{equation}
  defines a facet of
  $P^*$ if and only if $c_{p+1}-|W| \geq 2$ and $c_{p+1} < n$ or
  $c_{p+1}=n$ and $|W| = n-1$.
\end{Theorem}

\begin{proof}
  Assuming that $|W| + 1 = c_{p+1} < n$, we see that (\ref{CF}) is
  dominated by nonnegativity constraints $x_{ij} \geq 0$ for $(i,j)
  \in N \setminus W$. 
  When $c_{p+1} = n$ and $n - |W| \geq 2$, (\ref{CF}) is dominated by
  another inequality of the same form for some
  $W' \supset W$ with $|W'| = n-1$. Therefore, if inequalities (\ref{CF}) are not facet defining, then they are dominated by other inequalities of the IP-model that are facet defining for $P^*$.

  Suppose that $c_{p+1} - |W| \geq 2$ and $c_{p+1} < n$. By choice, $|W|
  \geq 3$ and $|N \setminus W| \geq 3$. Moreover, assume that the equation $bx = b_0$
is satisfied by all points that satisfy (\ref{CF}) at equality.
  Setting $\iota := c_{p+1}-|W|$, we will show that
  \begin{equation} \label{conc1}
\begin{array}{rcll}
    b_{1i} & = & \iota   & \forall \; i \in N \setminus \{1\} \\
    b_{i1} & = & \iota   & \forall \: i \in W \setminus \{1\}, \\
    b_{ij} & = & \kappa  & \forall \: i \in W \setminus \{1\}, j \in N \setminus \{1\}, \\
    b_{ij} & = & \lambda & \forall \: i \in N \setminus W, j \in N \setminus \{1\}, \\
    b_{i1} & = & \mu     & \forall \: i \in N \setminus W
  \end{array}
\end{equation}
for some $\kappa \neq 0, \lambda, \mu$.
  Then, considering a tight cycle of length $c_p$ and two tight
  cycles of length $c_{p+1}$, one using  an arc in $(N \setminus W: \{1\})$, the
  other not, yields the equation system
  \begin{equation*}
  \begin{array}{rcl}
    b_0 & = & 2 \iota + (c_p-2) \kappa \\
    b_0 & = &   \iota + (|W|-1) \kappa + (c_{p+1} - |W| -1) \lambda + \mu \\
    b_0 & = & 2 \iota + (|W|-2) \kappa + (c_{p+1}-|W|) \lambda
  \end{array}
  \end{equation*}
  which solves to 
 \begin{equation*}
  \begin{array}{rcl}
    b_0 & = & 2 \iota + (c_p-2) \kappa \\
    \mu & = & \iota + (\frac{|W|-c_p}{ |W|-c_{p+1}}-1) \kappa \\
    \lambda & = & \frac{|W|-c_p}{ |W|-c_{p+1}} \kappa.
  \end{array}
  \end{equation*}
  Thus, $bx=b_0$ is the equation
  \begin{equation*}
  \begin{array}{rcl}
  \iota x(\delta^{\mbox{\scriptsize out}}(1)) + \iota x(\delta^{\mbox{\scriptsize in}}(1)) +  (\frac{|W|-c_p}{
    |W|-c_{p+1}}-1) \kappa \sum\limits_{i \in N \setminus W} x_{i1} \\
  + \kappa \sum\limits_{i \in W \setminus \{1\}} x(\delta^{\mbox{\scriptsize out}}_1(i)) + 
  \frac{|W|-c_p}{ |W|-c_{p+1}} \kappa \sum\limits_{i \in N
    \setminus W} x(\delta^{\mbox{\scriptsize out}}_1(i)) & = &
  2 \iota + (c_p-2) \kappa,
  \end{array}
  \end{equation*}
where $\delta^{\mbox{\scriptsize out}}_1(i) := \delta^{\mbox{\scriptsize out}}(i) \setminus \{(i,1)\}$.
  Adding $\kappa -\iota$ times the equations $x(\delta^{\mbox{\scriptsize out}}(1))=1$ and
  $x(\delta^{\mbox{\scriptsize in}}(1))=1$ and multiplying the resulting equation with
  $- \frac{|W|-c_{p+1}}{\kappa}$, we see that $bx=b_0$ is equivalent to
  (\ref{CF}).

To show (\ref{conc1}), we may assume without loss of generality that
$2 \in W$ and $b_{1i}= c_{p+1}-|W|$, $i \in N \setminus \{1\}$, and
$b_{21}= c_{p+1}-|W|$, by Theorem \ref{equiv}. 
Next, let $\mathcal{R}$ be the set of subsets of $N$ of cardinality
$c_{p+1}$ that contain $W$, i.e.,  
\[\mathcal{R} := \{R \subset N| \:
  |R|= c_{p+1}, R \supset W\}.
\]
 For any $R \in \mathcal{R}$, the
  $c_{p+1}$-cycles on $R$ are tight tours on $R$. Theorem 23 of
  Gr\"otschel and Padberg \cite{GP} implies that there are $\tilde{\alpha}_i^R,
  \tilde{\beta}_i^R$ for $i \in R$ such that $b_{ij} = \tilde{\alpha}_i^R +
  \tilde{\beta}_j^R$ for all $(i,j) \in A(R)$. Setting
  \begin{equation}
  \begin{array}{rcll}
  \alpha_i^R &  := & \tilde{\alpha}_i^R - \tilde{\alpha}_1^R & (i
  \in R), \\
  \beta_i^R  &  := & \tilde{\beta}_i^R - \tilde{\alpha}_1^R & (i \in R),
  \end{array}
  \end{equation}
  yields $\alpha_i^R + \beta_j^R = b_{ij}$ for all $(i,j) \in A(R)$.
  Since $\alpha_1^R = 0$ and $b_{1i} = \iota$, it follows that
  $\beta_i^R = \iota$ for all $i \in R \setminus \{1\}$. 
In a similar manner one can show for any $S \in \mathcal{R}$
  the existence of $\alpha_i^S, \beta_i^S$ for $i \in S$ with
  $\alpha_1^S = 0$, $\beta_j^S = \iota$ for $j \in S \setminus
  \{1\}$, and $\alpha_i^S + \beta_i^S = b_{ij}$ for all $(i,j) \in
  A(S)$. This implies immediately that $\alpha_i^R = \alpha_i^S$ and
  $\beta_i^R = \beta_i^S$ for all $i \in R \cap S$. Thus,
there are $\alpha_i, \beta_i$ for all $i \in N$ such that
  $\alpha_1 = 0$, $\beta_i = \iota$ for $i \in N \setminus \{1\}$,
  and $b_{ij} = \alpha_i +\beta_j $ for all $(i,j) \in A$. 

  Next, consider a tight $c_{p}$-cycle
  that contains the arcs $(1,k), (k,j)$ but does
  not visit node $\ell$ for some $j,k,\ell \in W$. Replacing node $k$ by node $\ell$ yields another
  tight $c_{p}$-cycle, and therefore $b_{1k} + b_{kj} = b_{1\ell }+b_{\ell j }$,
  which implies that $\alpha_k = \alpha_{\ell }$ for all
  $k, \ell \in W \setminus \{1\}$. Thus, there is $\kappa$ such that
  $b_{ij}=\kappa$ for all $i \in W \setminus \{1\}$, $j \in N
  \setminus \{1\}$. Moreover, it follows immediately that
  $b_{i1}=\iota$ for all $i \in W \setminus \{1\}$. One can show analogously that $\alpha_i = \alpha_j$ for all $i,j
  \in N \setminus W$. This implies the existence of $\lambda, \mu$
  with $b_{ij} = \lambda$ for all $i \in N \setminus W$, $j \in N
  \setminus \{1\}$ and $b_{i1} = \mu$ for all $i \in N \setminus W$.

  Finally, when $|W|+1 = c_{p+1} = n$, we show that there are $n^2 -
  2n$ affinely independent points $x \in P^*$ satisfying (\ref{CF}) at
  equality. Without loss of generality, let $W=
  \{1,\dots,n-1\}$. Because each tour is tight with
  respect to 
  (\ref{CF}), it exist $n^2-3n+2$ linearly independent points $(x^r,y^r)
  \in Q := \{(x,y) \in P_{CL}^c(D_n) | x(\delta^{\mbox{\scriptsize out}}(1) = 1)\}$ with $y^r = 0$. Furthermore, 
consider the incidence vectors of the $n-2$ cycles $(1,2,\dots,c_p)$, $(1,3,4,\dots,c_p+1),\dots, (1,n-2,n-1,2,3,\dots,c_p-2)$, $(1,n-1,2,3,\dots,c_p-1)$.
The corresponding points in $Q$ are linearly independent and they are also linearly independent of the points $(x^r,y^r)$. Hence, (\ref{CF}) is also facet defining 
if $|W|+1 = c_{p+1} = n$.
\end{proof}


\begin{Theorem} \label{TRS}
  Let $D_n=(N,A)$ be the complete digraph on $n$ nodes, and let $1 \in W \subset N$ with $c_p < |W| < c_{p+1}$
for some $p \in \{1,\dots,m-1\}$. 
  The \emph{cardinality-subgraph inequality}
  \begin{equation} \label{RS}
    2x(A(W)) - (|W|-c_p-1) [x((W:N \setminus W)) + x((N \setminus W:W))]
    \leq 2 c_p
  \end{equation}
  is valid for $P^*$ and induces a facet of $P^*$
  if and only if $p+1 < m$ or $c_{p+1}=n = |W|+1$.
\end{Theorem}

\begin{proof}
A cycle of length less or equal to $ c_p$ uses at
most $c_p$ arcs of $A(W)$ and thus its incidence vector satisfies
(\ref{RS}). A cycle $C$ of length greater or equal to $c_{p+1}$ uses
at most $|W|-1$ arcs in $A(W)$ and if $C$ indeed visits any node in $W$, then it uses at least 2 arcs in $(W:N \setminus
W) \cup (N \setminus W:W)$ and hence, 
\[ 
\begin{array}{lrr}
\multicolumn{3}{l}{2 \chi^C(A(W)) - (|W|-c_p-1) [\chi^C ((W:N \setminus
  W)) + \chi^C((N \setminus W:W))] \hspace{1cm}} \\
& & \leq 2(|W|-1) - 2(|W|-c_p-1) = 2c_p.
\end{array} 
\] 
In particular, all cycles of feasible length that visit node $1$
satisfy (\ref{RS}). 

To prove that \eqref{RS} is facet defining,
assume that $p+1=m$ and $c_m < n$. When $c_{p+1}-c_p=2$ holds, then
(\ref{RS}) does not induce a facet of $P^*$ for the same reason as the
corresponding cardinality forcing inequality does not induce a facet
of $P^*$. Indeed, both inequalities define the same face. When
$c_{p+1}-c_p > 2$, then it is easy to see that the face induced by
(\ref{RS}) is a proper subset of the face defined by the cardinality
forcing inequality (\ref{CF}), and thus, it is not facet defining.
The same argumentation holds when $p+1=m$, $c_m =n$, and $n - |W| > 1$. 

To show that (\ref{RS}) defines a facet, when the conditions are satisfied, we
suppose that the equation $bx=b_0$ is satisfied by every $x \in P^*$
that satisfies (\ref{RS}) at equality. Using Theorem \ref{equiv} we
may assume that $b_{w1}= 2$ for some $w \in W$, $b_{1i}=2$ for all $i
\in W$, and $b_{iw} = -(|W|-c_p-1)$ for all $i \in N \setminus W$. 

Let $q,r \in N \setminus W$ be two nodes that are equal if
$c_{p+1}=|W|+1$ and otherwise different. Then, all $(q,r)$-paths of
length $|W|+1$ whose internal nodes are all the nodes of $W$ satisfies
the equation $bx = b_0$. (Note, in case $c_{p+1}=|W|+1$, the paths are
Hamiltonian cycles.) Thus, it exist $\alpha_q$, $\beta_r$, and 
$\alpha_j$, $\beta_j$ for $j \in W$ with 
\[
\begin{array}{rcll}
b_{qj} & = & \alpha_q+\beta_j & (j \in W)\\
b_{ir} & = & \alpha_i+\beta_r & (i \in W)\\  
b_{ij} & = & \alpha_i+\beta_j & ((i,j) \in A(W)). 
\end{array}
\]
Without loss of generality we may assume that $\beta_w=0$. Since
$b_{1j}=2$, it follows that $\alpha_1 = 2$, $\beta_j = 0$ for all
$j \in W \setminus \{1\}$, and $\alpha_q = |W|-c_p-1$. When $c_p=2$,
then the cycles $\{(1,j),(j,1)\}$ for $j \in W \setminus \{1\}$. When
$c_p \geq 3$, then consider a tight $c_p$-cycle that starts with
$(1,i),(i,j)$ and skips node $k$ for some $i,j,k \in W \setminus
\{1\}$. Replacing the arcs $(1,i),(i,j)$ by $(1,k),(k,j)$ yields
another tight $c_p$-cycle, and thus the equation
$b_{1i}+b_{ij}=b_{1k}+b_{kj}$. In either case, it follows that $b_{j1}
= 2$ for $j \in W \setminus \{1\}$ and there is $\lambda$
such that $b_{ij} = \lambda$ for all $(i,j) \in A(W \setminus
\{1\})$. Summarizing our intermediate results and adding further, easy
obtainable observations, we see that
\begin{equation}
\begin{array}{rcll} \label{eq3}
b_{1i} & = & 2 & (i \in W \setminus \{1\}) \\
b_{i1} & = & 2 & (i \in W \setminus \{1\}) \\
b_{ij} & = & \lambda & ((i,j) \in A(W \setminus \{1\})) \\
b_{qi} & = & -(|W|-c_p-1) & (i \in W \setminus \{1\}) \\
b_{q1} & = & -(|W|-c_p-1) + 2 - \lambda \\
b_{ir} & = & -(|W|-c_p-1)(\lambda-1) & (i \in W \setminus \{1\}) \\
b_{1r} & = & -(|W|-c_p-1)(\lambda-1) + 2 - \lambda \\
b_0 & = & 4 + (c_p-2) \lambda
\end{array}
\end{equation}
holds.

So, when $c_{p+1}=n$, we have $q=r$ and $N \setminus W = \{q\}$, and thus,
$bx = b_0$ is the equation 
\[
\begin{array}{rcl}
2 x(\delta^{\mbox{\scriptsize out}}(1)) - \lambda x_{1q}+2x(\delta^{\mbox{\scriptsize in}}(1))-\lambda x_{q1} +
\lambda x(A(W \setminus \{1\}))\\ 
-(|W|-c_p-1)x(\delta^{\mbox{\scriptsize out}}(q)) - (|W|-c_p-1)(\lambda-1) x(\delta^{\mbox{\scriptsize in}}(q))
& = &  4 + (c_p-2) \lambda.
\end{array}
\] 
Adding $(1-\frac{\lambda}{2})(|W|-c_p-1)$ times the equation
$x(\delta^{\mbox{\scriptsize out}}(q))-x(\delta^{\mbox{\scriptsize in}}(q))=0$ and $(\lambda-2)$ times the
equations $x(\delta^{\mbox{\scriptsize out}}(1))=1$ and $x(\delta^{\mbox{\scriptsize in}}(1))=1$, we see that
$bx=b_0$ is equivalent to (\ref{RS}), and hence (\ref{RS}) is facet
defining. 

Otherwise, that is, if $p+1 < m$, (\ref{eq3}) holds for each pair of
nodes $q,r \in N \setminus W$. Moreover, letting $k \neq l \in W
\setminus \{1\}$, it can be seen that every $(k,l)$-path $P$
of length $c_{p+1}-|W|+1$ or $c_m - |W|+1$ whose internal nodes are in
$N \setminus W$ satisfies the equation $bx=
-\lambda(|W|-c_p-1)$. Thus, there are $\pi_k, \pi_l$, and $\{\pi_j | j
\in N \setminus W\}$ such that 
\[
\begin{array}{rcll}
b_{kj} & = & \pi_k - \pi_j & (j \in N \setminus W)\\
b_{jl} & = & \pi_j - \pi_l & (j \in N \setminus W)\\  
b_{ij} & = & \pi_i - \pi_j & ((i,j) \in A(N \setminus W)).
\end{array}
\]
Since $b_{kj}=-(|W|-c_p-1)(\lambda-1)$ for $j \in N \setminus W$, it follows
that $\pi_i = \pi_j$ for all $i,j \in N \setminus W$ which implies
that $b_{ij} = 0$. Hence, $bx = b_0$ is the equation
\[
\begin{array}{rcl}
2 x(\delta^{\mbox{\scriptsize out}}(1))+2x(\delta^{\mbox{\scriptsize in}}(1))-\lambda \sum_{i \in N \setminus W}
(x_{1i}+x_{i1}) \\
+ \lambda x(A(W \setminus \{1\}))-(|W|-c_p-1)x((N \setminus W:W)) \\
- (|W|-c_p-1)(\lambda-1) x((W:N \setminus W))& = &  4 + (c_p-2) \lambda.
\end{array}
\] 
Adding $(1-\frac{\lambda}{2})(|W|-c_p-1)$ times the equation
\[x((N \setminus W:W))-x((W: N \setminus W))=0\] 
and $(\lambda-2)$ times the
equations $x(\delta^{\mbox{\scriptsize out}}(1))=1$ and $x(\delta^{\mbox{\scriptsize in}}(1))=1$, we see that
$bx=b_0$ is equivalent to (\ref{RS}), and hence (\ref{RS}) is facet
defining.  
\end{proof}

\subsection{Facets unrelated to cardinality restrictions}

\begin{Theorem} \label{Tnon}
Let $c \in \CS$ and $n \geq 4$. The \emph{nonnegativity constraint} 
\begin{equation} \label{nnc}
x_{ij} \geq 0
\end{equation}
defines a facet of $\CCPPon$ if and only if $c \neq (2,n)$ or
$c=(2,n)$, $n \geq 5$, and $(i,j)$ is an inner arc.   
\end{Theorem}

\begin{proof}
By Theorem 3.1 of \cite{Stephan}, (\ref{nnc}) defines a facet
of $P_{0,n-\mbox{\scriptsize path}}^{(k)}(\tilde{D}_n)$ whenever $4 \leq k \leq n-1$. Hence, Lemma \ref{L1}
implies that (\ref{nnc}) is facet defining for $\CCPPon$ if $n \geq 5$
and there is an index $p$ with $4 \leq c_p \leq n-1$. 
In case of $c \in \{(2,n),(3,n),(2,3,n)\}$, see~\cite{KS}.
\end{proof}

\begin{Theorem} \label{Tdegree}
Let $c \in \CS$, $n \geq 4$, and $i$ be an internal node of $\tdn$. The
degree constraint  
\begin{equation} \label{degree}
x(\delta^{\mbox{\scriptsize out}}(i)) \leq 1
\end{equation}
induces a facet of $\CCPPon$ unless $c=(2,n)$.
\end{Theorem}

\begin{proof} 
When $n \geq 5$ and $4 \leq c_p \leq n-1$ for some index $p$,
(\ref{degree}) can be shown to induce a facet of $\CCPPon$ using Lemma
\ref{L1} of this paper and Theorem 3.2 of \cite{Stephan},
saying that (\ref{degree}) induces a facet of
$P_{0,n-\mbox{\scriptsize path}}^{(c_p)}(\tdn)$.
In case of $c \in \{(2,n),(3,n),(2,3,n)\}$, see~\cite{KS}.
\end{proof}


\begin{Theorem} \label{Tosmc}
Let $c = (c_1,\dots,c_m) \in \CS$, $n \geq 4$, $S \subset \tnn$, $0,n \in
S$, and $v \in \tnn \setminus S$. The \emph{one-sided min-cut inequality}
\begin{equation} \label{osmc}
x((S:\tnn \setminus S)) - x(\delta^{\mbox{\scriptsize in}}(v)) \geq 0
\end{equation}
induces a facet of $\CCPPon$ if and only if $|\tnn \setminus S| \geq 2$,
$|S| \geq c_1+1$, and $c \neq (2,n)$. 
\end{Theorem}

\begin{proof} 
\emph{Necessity.}
When $\tnn \setminus S = \{v\}$, (\ref{osmc}) becomes the trivial
inequality $0x \geq 0$, and thus it is
not facet defining. When $|S| \leq c_1$, all feasible $(0,n)$-paths 
$P$ satisfy $|P \cap (S:\tnn \setminus S)| \geq 1$, and hence,
(\ref{osmc}) can be obtained by summing up the inequality $x((S:\tnn
\setminus S)) \geq 1$ and the degree constraint $-x(\delta^{\mbox{\scriptsize in}}(v)) \geq
-1$. When $c=(2,n)$, see~\cite{KS}.

\emph{Sufficiency.} 
By Theorem 3.4 of \cite{Stephan}, (\ref{osmc}) induces a facet of
$P_{0,n-\mbox{\scriptsize path}}^{(k)}(\tdn)$ for $4 \leq k \leq n-2$ if
and only if $|S| \geq k+1$ and $|\tnn \setminus S| \geq 2$.  Hence, when
$|S| \geq c_i+1$ for some index $i \in \{1,\dots,m\}$ with $c_i \geq
4$ and $|\tdn \setminus S| \geq 2$, inequality (\ref{osmc}) is facet defining for $\CCPPon$ by
applying Lemma \ref{L1}. In particular, this finishes the proof when
$i=1$. Note that in case $i=m$, $c_i \geq 4$ and
$|S| \geq c_i+1$ imply $4 \leq c_m \leq n-2$, since $|S| \leq n-1$.
When $c_1 = 2$ or $c_1 = 3$, see~\cite{KS}.
\end{proof}

We introduce a further class of inequalities whose undirected pendants
we need later for the characterization of the integer points of
$\uCCCPn$.

\begin{Theorem} \label{Tmincut}
Let $c \in \CS$, $n \geq 4$, $S \subset \tnn$, and $0,n \in
S$. The \emph{min-cut inequality}
\begin{equation} \label{mc}
x((S:\tnn \setminus S)) \geq 1
\end{equation}
is valid for $\CCPPon$ if and only if $|S| \leq c_1$ and facet
defining for it if and only if $3 \leq |S| \leq c_1$ and $|\tnn
\setminus S| \geq 2$. 
\end{Theorem}

\begin{proof}
When $c \neq (3,n)$, the theorem follows from Theorem 3.3 of
\cite{Stephan}, Lemma \ref{L1}, and the fact that $m \geq 2$. When
$c=(3,n)$,  see~\cite{KS}.
\end{proof}

\subsection{Inequalities specific to odd or even paths}

\begin{Theorem} \label{TOP}
Let $c=(c_1,\dots,c_m)$ be a cardinality sequence with $m \geq 2$,
$c_1 \geq 2$, and
$c_p$ even for $1 \leq p \leq m$, and let $\tnn = S \; \dot{\cup} \;  T$ be a partition of $\tnn$
with $0 \in S$, $n \in T$. The \emph{odd path exclusion constraint}
\begin{equation} \label{ocec}
x(\tann(S))+ x(\tann(T)) \geq 1
\end{equation}
is valid for $\CCPPon$ and defines a facet of  $\CCPPon$ if and only
if (i) $c_1=2$ and $|S|, |T| \geq \frac{c_2}{2}+1$, or (ii) $c_1 \geq 4$ and 
$|S|,|T| \geq \frac{c_2}{2}$. 
\end{Theorem}

\begin{proof}
Clearly,
each $(0,n)$-path of even length uses at least one arc in $\tann(S) \cup
\tann(T)$. Thus, inequality (\ref{ocec}) is valid.

When $|S|$ or $|T|$ is less than
$\frac{c_2}{2}$, then there is no $(0,n)$-path of length $c_p$, $p
\geq 2$, that satisfies (\ref{ecec}) at equality which implies that (\ref{ecec}) cannot be facet defining for $\CCPPon$. Thus $|S|, |T| \geq
\frac{c_2}{2}$ holds if \eqref{ocec} is facet defining. For $c_1=2$ we have to require
even $|S|, |T| \geq \frac{c_2}{2}+1$. For the sake of contradiction
assume w.l.o.g. that $|S|=\frac{c_2}{2}$. Then follows $|T| \geq
\frac{c_2}{2}+1$. However, for an inner arc $(i,j) \in \tann(S)$ there is
no tight $(0,n)$-path of cardinality $c_2$ that uses $(i,j)$. 

Next, let (i) or (ii) be true. The conditions imply that for $p=1$ or $p=2$ $c_p \geq 4$ and $|S|,|T| \geq \frac{c_p}{2}+1$ holds. Restricted to the polytope 
$P_{0,n-\mbox{\scriptsize path}}^{(c_p)}(\tdn)$ inequality (\ref{ocec}) is equivalent to the max-cut inequality $x((S:T)) \leq \frac{c_p}{2}$ which were shown to be facet defining for $P_{0,n-\mbox{\scriptsize path}}^{(c_p)}(\tdn)$ (see Theorem 3.5 of \cite{Stephan}). Thus there are $n^2-2n-1$ linearly independent points in $\CCPPon \cap P_{0,n-\mbox{\scriptsize path}}^{(c_p)}(\tdn)$ satisfying (\ref{ocec}) at equality. Moreover, the conditions ensure that there is also a tight $(0,n)$-path of cardinality $c_q$, where $q=3-p$. By Lemma \ref{L1} (i), the incidence vector of this path is affinely independent of the former points, and hence, (\ref{ocec}) defines a facet of $\CCPPon$.
\end{proof}

\begin{Theorem} \label{TEP}
Let $c=(c_1,\dots,c_m)$ be a cardinality sequence with $m \geq 2$,
$c_1 \geq 3$, and
$c_p$ odd for $1 \leq p \leq m$, and let $\tnn = S \; \dot{\cup} \;  T$ be a partition of $\tnn$
with $0,n \in S$. The \emph{even path exclusion constraint}
\begin{equation} \label{ecec}
x(\tann(S))+ x(\tann(T)) \geq 1
\end{equation}
is valid for $\CCPPon$ and defines a facet of  $\CCPPon$ if and only
if (i) $c_1=3$, $|S|-1 \geq \frac{c_2+1}{2}$, and $|T| \geq
\frac{c_2-1}{2}$, or (ii) $c_1 \geq 5$ and $\min (|S|-1,|T|) \geq
\frac{c_2-1}{2}$. 
\end{Theorem}

\begin{proof}
Up to one special case, Theorem \ref{TEP} can be proved quite
similarly as Theorem \ref{TOP}. Hence, we skip the proof here and
refer the interested reader to~\cite{KS}.
\end{proof}

\begin{Theorem} \label{TMCF}
Let $D_n=(N,A)$ be the complete digraph on $n \geq 6$ nodes and
$c=(c_1,\dots,c_m)$ a cardinality sequence with $m \geq 3$, $c_1 \geq
2$, $c_m \leq n$, and $c_{p+2}=c_{p+1}+2=c_{p}+4$ for some $2 \leq p \leq m-2$.
Moreover, let $N=P \; \dot{\cup} \;  Q \; \dot{\cup} \; \{r\}$ be a partition of $N$, where $P$ contains
node $1$ and satisfies $|P|=c_p+1=c_{p+1}-1$. Then the \emph{modified
  cardinality forcing inequality}  
\begin{equation} \label{MCF}
\sum_{v \in P} x(\delta^{\mbox{\scriptsize out}}(v)) - \sum_{v \in Q} x(\delta^{\mbox{\scriptsize out}}(v)) +
x((Q:\{r\})) - x((P:\{r\})) \leq c_p
\end{equation}
defines a facet of $P^* = \{x \in \CCCPn | x(\delta^{\mbox{\scriptsize out}}(1))=1$. 
\end{Theorem}

\begin{proof}
The arcs that are incident with node $r$
have coefficients zero. Let $C$ be a cycle that visits node $1$ and is of
feasible length. If $C$ does not visit node $r$, $C$ satisfies clearly
(\ref{MCF}), since the restriction of (\ref{MCF}) to the arc set $A(N
\setminus \{r\})$ is an ordinary cardinality forcing inequality
(\ref{CF}). When $C$ visits node $r$ and uses at most $c_p$ arcs whose
corresponding coefficients are equal to one, then $C$ satisfies also
(\ref{MCF}), since all those coefficients that are not equal to 1 are $0$
or $-1$. So, let $C$ with $|C| \geq c_{p+1} $visit node $r$ and use as
many arcs whose 
corresponding coefficients are equal to one as possible. That are
exactly $|P|$ arcs which are contained in $A(P) \cup (P:Q)$. But then
$C$ must use at least one arc in $A(Q) \cup (Q:P)$ whose coefficient
is $-1$. Hence, also in this case $C$ satisfies (\ref{MCF}), which
proves the validity of (\ref{MCF}).

To show that \eqref{MCF} is facet defining,
suppose that the equation $bx = b_0$ is satisfied by all points that
satisfy (\ref{MCF}) at equality. By Theorem \ref{equiv}, we may
assume that $b_{1r}=b_{r1}=0$ and $b_{1i}=1$ for $i \in N \setminus
\{1,r\}$. By considering the $c_{p+1}$-cycles with respect to $P \cup
\{j\}$ for $j \in N \setminus P$, one can show along the lines of the
proof of Theorem \ref{TCF} that there are $\alpha_k$, $\beta_k$, $k
\in N$, with $b_{ij}=\alpha_i+\beta_j$ for all $(i,j) \in A$, $\alpha_1=0$,
$\beta_r=0$, and $\beta_j=1$. In particular, when $c_{p}=2$, the tight 2-cycles
$\{(1,i),(i,1)\}$, $i \in P$ yield $\alpha_k=\alpha_\ell$ for
$k,\ell \in P \setminus \{1\}$. Otherwise one can show as in the proof of Theorem
\ref{TCF} that $\alpha_k = \alpha_l$ for all $k,l \in P \setminus
\{1\}$. Thus, there is $\kappa$ such that $\alpha_i= \kappa$ for $i \in P$, $i \neq 1$. This in turn implies that there is $\lambda$ with $\alpha_j= \lambda$ for $j \in Q$ by considering tight $c_{p+1}$-cycles. 
Then, the equation $b_{r1}=0$, a tight cycle of length $c_p$, and two tight
cycles of length $c_{p+1}$, one visiting node $r$, the other a node $j \in Q$, yield the equation system
  \begin{equation*}
  \begin{array}{rcl}
b_{r1} & = & 0\\
    b_0 & = & (c_p-1)(\kappa+1) +\beta_1 \\
    b_0 & = & c_p(\kappa+1)\\
    b_0 & = & c_p(\kappa+1)+\lambda+\beta_1+1\\
  \end{array}
  \end{equation*}
  which solves to 
 \begin{equation*}
  \begin{array}{rcl}
    b_0 & = & c_p(\kappa+1)\\
    \lambda & = &- \kappa -2\\
    \beta_1 & = & \kappa+1\\
\alpha_r & = & -\kappa -1.
  \end{array}
  \end{equation*}

Next, consider for $i \in P \setminus \{1\}$, $j,k \in Q$ a $c_{p+2}$-cycle $C$ that starts in node 1, then visits all nodes in $P \setminus \{1,i\}$, followed by the nodes $j$, $r$, $i$, $k$, and finally returns to 1. Since $C$ is tight, we can derive the equation 
$$1+(c_p-1)(\kappa+1)+ b_jr+(\alpha_r+1)+(\kappa+1)+(\lambda+\beta_1)=b_0$$
which solves to $b_{jr}=\kappa$. By considering further tight $c_{p+2}$-cycles one can deduce that $b_{ri}=-\kappa$ for $i \in Q$ and $b_{jk}=-\kappa-1$ for $(j,k) \in A(Q)$. Thus, $bx=b_0$ is the equation
\[
\begin{array}{rcl}
x(\delta^{\mbox{\scriptsize out}}(1) \setminus \{(1,r)\})-x((Q: \{1\}))+ (2\kappa+1)x((P \setminus \{1\}: \{1\}))\\
+ (\kappa+1) \sum_{i \in P \setminus \{1\}} x(\delta^{\mbox{\scriptsize out}}(i) \setminus \{(i,1),(i,r)\})\\
- (\kappa+1)  \sum_{i \in Q} x(\delta^{\mbox{\scriptsize out}}(i) \setminus \{(i,1),(i,r)\})\\
- \kappa x(\delta^{\mbox{\scriptsize out}}(r) \setminus \{(r,1)\})
+ \kappa x(\delta^{\mbox{\scriptsize in}}(r) \setminus \{(1,r)\})& = & c_p(\kappa+1).
\end{array}
\]
Adding $\kappa$ times the equations $x(\delta^{\mbox{\scriptsize out}}(1))-x(\delta^{\mbox{\scriptsize in}}(1))=0$ and 
 $x(\delta^{\mbox{\scriptsize out}}(r))-x(\delta^{\mbox{\scriptsize in}}(r))=0$, we see that $bx=b_0$ is equivalent to (\ref{MCF}), and hence, (\ref{MCF}) defines a facet.
\end{proof}

\subsection{Separation}

All inequalities of the IP-model  \eqref{pathmodel} as well as the
min-cut inequalities \eqref{mc} and the modified cardinality forcing
inequalities \eqref{MCF} can be separated in polynomial time.  For the
one-sided min-cut inequalities (\ref{osmc}), separation consists
in finding a minimum $\{0,n\}-l$-cut in $\tdn$ for each node $l \in \tnn
\setminus \{0,n\}$. The cardinality forcing inequalities can be
separated with a greedy algorithm. To this end, let $x^* \in
\mathbb{R}^{\tann}_+$ be a fractional point. Set $y^*_i := x^*(\delta^{\mbox{\scriptsize out}}(i))$
for $i=0,\dots,n-1$, and apply the greedy separation algorithm 8.27 of
Gr\"otschel \cite{Groetschel} on input data $y^*, \tnn$, and $c$. To
separate the modified cardinality forcing inequalities this algorithm
can be applied $n-1$ times as subroutine, namely: for each internal
node $r$ of $\tnn$, apply it on the subgraph induced by $\tnn \setminus
\{r\}$.

Next, the separation problem for the odd (even) path exclusion constraints is equivalent to the maximum cut problem which is known to be NP-hard. Turning to the cardinality-subgraph inequalities \eqref{RS}, it seems to be very unlikely that there is a polynomial time algorithm that solves the separation problem for this class of inequalities. 
Assume that there is given an instance $(D'=(N',A'),c=(c_1,\dots,c_m),
x^*)$ of the separation problem, where $x^* \in A'$ is a fractional
point satisfying $x^*(\delta^{\mbox{\scriptsize out}}(1))=1$. (We
consider the separation problem for $P^*$.) In the special case of
$m=2$ and $c_m=c_2-c_1=2$ the separation problem for the inequalities
\eqref{RS} and $x^*$ reduces to find a subset $W^*$ of $N'$ of
cardinality $k:=c_1+1$ such that $1 \in W^*$ and $x^*(A'(W^*)) >
2c_p$. This problem can be tackled on the underlying graph
$G'=(N',E')$ with edge weights $w_e:=x^*_{ij}+x^*_{ji}$ for $e=[i,j]
\in E'$, where $x^*_{ij}$ is set to zero if the arc $(i,j)$ is not in
$A'$. The associated optimization problem $\max w(E'(W)), W \subseteq
N', 1 \in W, |W|=k$, is a variant of the weighted version of the
densest $k$-subgraph problem which is known to be NP-hard (see Feige and
Seltser~\cite{FS}).

\section{Facets of the other polytopes}

In this section, we derive facet defining inequalities for related
polytopes mentioned in the introduction from facet defining
inequalities for the cardinality constrained path polytope
$\CCPPon$. 

\subsection{Facets of the directed cardinality constrained cycle
  polytope}

\begin{Corollary} \label{CfacetsCCCP}
Let $D_n=(N,A)$ be the complete digraph on $n \geq 3$ nodes and
$c=(c_1,\dots,c_m)$ a cardinality sequence with $m \geq 2$ and $c_1
\geq 2$.  Then the following statements hold:\\[0.5em]
(a) The nonnegativity constraint $x_{ij} \geq 0$ defines a facet
  of $\CCCPn$.\\[0.5em]
(b) The degree constraint $x(\delta^{\mbox{\scriptsize out}}(i)) \leq 1$ defines a
  facet of $\CCCPn$ for every $i \in N$. \\[0.5em]
(c) Let $S$ be a subset of $N$ with $2 \leq |S| \leq n-2$, let
  $v \in S$ and $w \in N \setminus S$. The \emph{multiple cycle
    exclusion constraint}
\begin{equation} \label{mcec}
x(\delta^{\mbox{\scriptsize out}}(v))+x(\delta^{\mbox{\scriptsize out}}(w)) - x((S:N \setminus S)) \leq 1
\end{equation}
induces a facet of $\CCCPn$ if and only if $|S|,|N \setminus S| \geq
c_1$ and $c \notin \{(2,3),(2,n)\}$.\\[0.5em]
(d) For any $S \subset N$ with $|S|,|N \setminus S| \leq c_1-1$, the min-cut inequality
\begin{equation} \label{cmc}
x((S:N \setminus S)) \geq 1
\end{equation}
is valid for $\CCCPn$ and induces a facet of $\CCCPn$ if and only if
$|S|,|N \setminus S| \geq 2$. \\[0.5em]
(e) Let $S$ be a subset of $N$ and $j \in N \setminus S$. The
  one-sided min-cut inequality 
\begin{equation} \label{cosmc} 
x((S:N \setminus S)) -x(\delta^{\mbox{\scriptsize out}}(j)) \geq 0
\end{equation}
defines a facet of $\CCCPn$ if and only if
$|S| \geq c_1$ and $ 2 \leq |N \setminus S| \leq c_1-1$.\\[0.5em]
(f) The cardinality bound $x(A) \geq c_1$
  defines a facet of $\CCCPn$ if and only if $c_1=3$ and $n \geq 5$ or
  $4 \leq c_1 \leq n-1$. Analogously,  $x(A) \leq c_m$
  defines a facet of $\CCCPn$ if and only if $c_m=3$ and $n \geq 5$ or
  $4 \leq c_m \leq n-1$.\\[0.5em]
(g) Let $W$ be a subset of $N$ with $c_p < |W| <
  c_{p+1}$ for some $p \in \{1,\dots,m-1\}$. The cardinality-forcing
  inequality (\ref{CF}) defines a facet of $\CCCPn$ if and only if
  $c_{p+1}-|W| \geq 2$ and $c_{p+1} < n$ or $c_{p+1}=n$ and $|W| =
  n-1$. \\[0.5em]
(h) Let $W$ be a subset of $N$ such that $c_p < |W| < c_{p+1}$ holds
for some $p \in \{1,\dots,m-1\}$. The cardinality-subgraph inequality (\ref{RS}) is valid for $\CCCPn$ and induces a facet of $\CCCPn$
if and only if $p+1 < m$ or $c_{p+1}=n = |W|+1$. \\[0.5em]
(i) Let $c=(c_1,\dots,c_m)$ be a cardinality sequence with $m \geq 2$,
$c_1 \geq 2$, and
$c_p$ even for $1 \leq p \leq m$, and let $N = S \; \dot{\cup} \;  T \; \dot{\cup} \;  \{n\}$ be a partition of $N$. The \emph{odd cycle exclusion constraint}
\begin{equation} \label{opec}
x(A(S))+ x(A(T)) + x((T: \{n\}))- x((\{n\}:T)) \geq 0
\end{equation}
is valid for $\CCCPn$ and defines a facet of  $\CCCPn$ if and only
if ($\alpha$) $c_1=2$ and $|S|, |T| \geq \frac{c_2}{2}$, or ($\beta$) $c_1 \geq 4$ and 
$|S|,|T| \geq \frac{c_2}{2}-1$. \\[0.5em]
(j) Let $c=(c_1,\dots,c_m)$ be a cardinality sequence with $m \geq 2$,
$c_1 \geq 3$, and
$c_p$ odd for $1 \leq p \leq m$, and let $N = S \; \dot{\cup} \;  T$ be a partition of $N$. The \emph{even cycle exclusion constraint}
\begin{equation} \label{epec}
x(A(S))+ x(A(T)) \geq 1
\end{equation}
is valid for $\CCCPn$ and defines a facet of  $\CCCPn$ if and only
if $|S|, |T| \geq \frac{c_2-1}{2}$. \\[0.5em]
(k) Let $c=(c_1,\dots,c_m)$ be a cardinality sequence with $m \geq 3$, $c_1 \geq
2$, $c_m \leq n$, $n \geq 6$, and $c_{p+2}=c_{p+1}+2=c_{p}+4$ for some $2 \leq p \leq m-2$.
Moreover, let $N=P \; \dot{\cup} \;  Q \; \dot{\cup} \; \{r\}$ be a partition of $N$, with $|P|=c_p+1=c_{p+1}-1$. Then the modified
  cardinality forcing inequality \eqref{MCF} defines a facet of $\CCCPn$. 
\end{Corollary}

\begin{proof}
(a) When $n \leq 4$, the statement can be verified using a computer program.
When $c=(2,3)$ and $n \geq 5$, we apply Theorem 10 of Hartmann and
\"Ozl\"uk which says that  $x_{ij} \geq 0$ defines a facet of
$P_C^{(p)}(D_n)$ whenever $p \geq 3$ and $n \geq p+1$. Thus, there are
$n^2-2n$ 3-cycles satisfying $x_{ij} \geq 0$ at equality. Together with
Lemma \ref{L1} applied on these tight 3-cycles and any 2-cycle not
using arc $(i,j)$, we get the desired result. The remainder statements
of (a) follow by application of Theorem \ref{Tnon} and Theorem
\ref{lifting}.
\\[0.5em]
(b) First, when $c=(2,3)$ one can show along the lines of the proof to
Proposition 5 of Balas and Oosten \cite{BO} that $x(\delta^{\mbox{\scriptsize out}}(i)) \leq 1$
defines a facet of $\CCCPn$. Next, when $(2,3) \neq c \neq (2,n)$, the
degree constraint can be shown to induce a facet using  theorems
\ref{Tdegree} and \ref{lifting}. Finally, when $c=(2,n)$, 
see~\cite{KS}.
\\[0.5em]
(c) Supposing that $c=(2,3)$, the inequality (\ref{mcec}) is
dominated by the nonnegativity constraint $x_{ij} \geq 0$ for any arc
$(i,j) \in (S: N \setminus S) \cup (N \setminus S: S)$ that is neither
incident with $v$ nor with $w$. 
Next, suppose that $c=(2,n)$. Inequality (\ref{mcec}) is equivalent to the
subtour elimination constraint $x(A(S)) \leq |S|-1$ with respect to
the ATSP $P_C^{(n)}(D_n)$. Thus, we have $n^2-3n+1$ tours satisfying
(\ref{mcec}) at equality. But we have only $n-1$ tight 2-cycles, and
consequently, (\ref{mcec}) does not induce a facet.
Next, if $|S| \leq c_1-1$, then (\ref{mcec}) is the sum of the valid
inequalities $x(\delta^{\mbox{\scriptsize out}}(v)) - x((S:N \setminus S)) \leq 0$ and
$x(\delta^{\mbox{\scriptsize out}}(w)) \leq 1$. Finally, if $|N \setminus S| \leq c_1-1$,
then (\ref{mcec}) is the sum of the inequalities $x(\delta^{\mbox{\scriptsize out}}(w)) -
x((S:N \setminus S)) \leq 0$ and $x(\delta^{\mbox{\scriptsize out}}(v)) \leq 1$ (cf.
Hartmann and \"Ozl\"uk \cite[p. 162]{HO}).

Suppose that the conditions in (c) are satisfied. First, consider the
inequality (\ref{mcec}) on the polytope $Q:=\{x \in \CCCPn :
x(\delta^{\mbox{\scriptsize out}}(1))=1\}$ which is isomorphic to the path polytope $\CCPPon$. Then,  (\ref{mcec}) is equivalent to the one-sided
min-cut inequality (\ref{osmc}) which defines a facet of $Q$ by
Theorem \ref{Tosmc}. Thus, also (\ref{mcec}) defines a facet of $Q$. Now,
by application of Theorem \ref{lifting} on $Q$ and  (\ref{mcec}) we
obtain the desired result. 
(When $c_1 \geq 4$, then the statement can be proved also with Theorem
14 of Hartmann and \"Ozl\"uk \cite{HO}. 
\\[0.5em]
(d) Assuming $|S|=1$ or $|N \setminus S|=1$ implies that (\ref{cmc}) is an implicit equation. 
So, let $|S|,|N \setminus S| \geq 2$ which implies that $c_1 \geq 3$. From Theorem \ref{Tmincut} follows that
(\ref{cmc}) defines a facet of $Q:= \{x \in \CCCPn : x(\delta^{\mbox{\scriptsize out}}(i)) =1\}$, and hence, by Theorem \ref{lifting}, 
it defines also a facet of $\CCCPn$.
\\[0.5em]
(e) When $|N \setminus S| \geq c_1$, (\ref{cosmc}) is obviously not
valid. When $|N \setminus S|=1$,  (\ref{cosmc}) is the flow constraint
$x(\delta^{\mbox{\scriptsize in}}(j))- x(\delta^{\mbox{\scriptsize out}}(j)) = 0$. 
When $|S| \leq c_1-1$ and $|N \setminus S| \leq c_1-1$, (\ref{cosmc})
is the sum of the valid inequalities $x((S: N \setminus S)) \geq 1$
and $-x(\delta^{\mbox{\scriptsize out}}(j)) \geq -1$.

Suppose that $|S| \geq c_1$ and $ 2 \leq |N \setminus S| \leq
c_1-1$. Then in particular $c_1 \geq 3$ holds. For any node $i \in S$,
(\ref{cosmc}) defines a facet of $Q:= \{x \in \CCCPn :
x(\delta^{\mbox{\scriptsize out}}(i))=1\}$, by Theorem \ref{Tosmc}. Applying Theorem
\ref{lifting} we see that therefore (\ref{cosmc}) defines also a facet
of $\CCCPn$.
\\[0.5em]
(f) Since $\dim \{x \in \CCCPn : x(A)=c_i\} = \dim P_C^{(c_i)}(D_n)$,
the claim follows directly from Theorem 1 of Hartmann and \"Ozl\"uk
\cite{HO}.   
\\[0.5em]
(g)-(i) Necessity can be proved as in the corresponding part of
the proof to Theorem \ref{TCF} (\ref{TRS}, \ref{TOP}) while suffiency can
be shown by applying Theorem \ref{lifting} on Theorem \ref{TCF}
(\ref{TRS}, \ref{TOP}). 
\\[0.5em]
(j) By Theorem 15 of Hartmann and \"Ozl\"uk, \eqref{epec} defines a facet of $P_C^{(c_1)}(D_n)$. Moreover, the cardinality conditions for $S$ and $T$ ensure that there is a tight cycle of cardinality $c_2$, and hence, by Lemma \ref{L1}, \eqref{epec} defines a facet of $\CCCPn$. 
\\[0.5em]
(k) Apply Theorem \ref{lifting} on Theorem \ref{TMCF}.
\end{proof}

\subsection{Facets of the undirected cardinality constrained cycle
  polytope} 

In this section, we consider the undirected cardinality constrained
cycle polytope $\uCCCPn$ defined on the complete graph $K_n=(N,E)$,
where $c$ is a cardinality sequence with $3 \leq c_1 < \dots < c_m
\leq n$ and $m \geq 2$. It was shown in \cite{KMV} and \cite{MN2} that
$\dim P_C^{(p)}(K_n) = |E|-1$ for $3 \leq p \leq n-1$ and $n \geq
5$. Thus, it is easy to verify that
$\dim \uCCCPn = |E|=n(n-1)/2$ for all $n \geq 4$, since $m \geq
2$. Note, in case of $n=4$, $\uCCCPn = P_C(K_n)$, and by Theorem 2.3
of Bauer \cite{Bauer}, $\dim P_C(K_4)=6=|E|$. 

Facet defining inequalities for $\uCCCPn$ can be derived directly from
the inequalities mentioned in Corollary \ref{CfacetsCCCP} (b)-(h),
since these inequalities are equivalent to symmetric inequalities. A
valid inequality $cx \leq \gamma$ for $\CCCPn$ is said to be
\emph{symmetric} if $c_{ij}=c_{ji}$ holds for all $i < j$.  Due to the
flow conservation constraints, it is equivalent to a symmetric inequality if and only if the system $t_i-t_j=c_{ij}-c_{ji}$ is consistent (see Hartmann and \"Ozl\"uk~\cite{HO} and Boros et al~\cite{BHHS}). One can
show that the undirected counterpart $\sum_{1 \leq i < j \leq n}
c_{ij} y_{ij}$ of a symmetric inequality $c x \leq \gamma$ is valid for
$\uCCCPn$. Moreover, it induces a facet of $\uCCCPn$ if $cx \leq
\gamma$ induces a facet of $\CCCPn$. This follows from an argument of
Fischetti \cite{Fischetti}, originally stated for the ATSP and STSP,
which is also mentioned in
Hartmann and \"Ozl\"uk \cite{HO} in the context of directed and undirected $p$-cycle
polytopes $P_C^{(p)}(D_n)$ and $P_C^{(p)}(K_n)$. 

\begin{Corollary} \label{CfacetsuCCCP}
Let $K_n=(N,E)$ be the complete graph on $n \geq 3$ nodes and
$c=(c_1,\dots,c_m)$ a cardinality sequence with $m \geq 2$ and $c_1
\geq 3$.  Then holds:
\\[0.5em]
(a) For any $e \in E$, the nonnegativity constraint $y_e \geq
  0$ defines a facet of $\uCCCPn$ if and only if $n \geq 5$. \\[0.5em]
(b) The degree constraint $y(\delta(i)) \leq 2$ defines a facet
  of $\uCCCPn$ for every $i \in N$. \\[0.5em]
(c) Let $S$ be a subset of $N$ with $c_1 \leq |S| \leq n-c_1$, let
  $v \in S$ and $w \in N \setminus S$. Then, the two-sided min-cut inequality
\begin{equation} \label{umcec}
y(\delta(v))+y(\delta(w)) - y((S:N \setminus S)) \leq 2
\end{equation}
induces a facet of $\uCCCPn$.\\[0.5em]
(d) For any $S \subset N$ with $|S|,|N \setminus S| \leq c_1-1$, the min-cut inequality
\begin{equation} \label{ucmc}
y((S:N \setminus S)) \geq 2
\end{equation}
is valid for $\uCCCPn$ and induces a facet of $\uCCCPn$ if and only if
$|S|,|N \setminus S| \geq 2$.\\[0.5em] 
(e) Let $S$ be a subset of $N$ and $j \in N \setminus S$. The
  one-sided min-cut inequality 
\begin{equation} \label{ucosmc} 
y((S:N \setminus S)) -y(\delta(j)) \geq 0
\end{equation}
defines a facet of $\uCCCPn$ if and only if
$|S| \geq c_1$ and $ 2 \leq |N \setminus S| \leq c_1-1$.\\[0.5em]
(f) The cardinality bound $y(E) \geq c_1$ defines a facet of
  $\uCCCPn$. The cardinality bound $y(E) \leq c_m$ defines a facet of
  $\uCCCPn$ if and only if $c_m < n$.\\[0.5em]
(g) Let $W$ be a subset of $N$ with $c_p < |W| <
  c_{p+1}$ for some $p \in \{1,\dots,m-1\}$. The cardinality-forcing
  inequality
\begin{equation}
(c_{p+1}-|W|) \sum_{i \in W} y(\delta(i)) - (|W|-c_p)\sum_{i \in N
  \setminus W} y(\delta(i)) \leq 2 c_p( c_{p+1}-|W|) 
\end{equation}
defines a facet of $\uCCCPn$ if and only if $c_{p+1}-|W| \geq 2$ and $c_{p+1} < n$ or $c_{p+1}=n$ and $|W| = n-1$.\\[0.5em]
(h) Let $W$ be a subset of $N$ such that $c_p < |W| < c_{p+1}$ holds
for some $p \in \{1,\dots,m-1\}$. The cardinality-subgraph inequality
\begin{equation}  
2y(E(W)) - (|W|-c_p-1) y((W:N \setminus W))
    \leq 2 c_p
\end{equation}
 is valid for $\uCCCPn$ and induces a facet of $\uCCCPn$
if and only if $p+1 < m$ or $c_{p+1}=n = |W|+1$.\\[0.5em]
(i) Let $c=(c_1,\dots,c_m)$ be a cardinality sequence with $m \geq 2$,
$c_1 \geq 3$, and
$c_p$ odd for $1 \leq p \leq m$, and let $N = S \; \dot{\cup} \;  T$ be a partition of $N$. The even cycle exclusion constraint
\begin{equation}
y(E(S))+ y(E(T)) \geq 1
\end{equation}
is valid for $\uCCCPn$ and defines a facet of  $\uCCCPn$ if and only
if $|S|, |T| \geq \frac{c_2-1}{2}$. 
\end{Corollary}

\begin{proof}
(a) When $n \leq 5$ the statement can be verified using a computer
program. When $n \geq 6$, the claim follows from Proposition 2 of
Kovalev, Maurras, and Vax\'{e}s \cite{KMV}, Proposition 2 of Maurras and
Nguyen \cite{MN2}, and the fact that $m \geq 2$.

(b)-(i) All directed inequalities occurring in Corollary
\ref{CfacetsCCCP} (b)-(h) and (j) are equivalent to symmetric
inequalities. For example, the degree constraint $x(\delta^{\mbox{\scriptsize out}}(i)) \leq
1$ is equivalent to $x(\delta^{\mbox{\scriptsize out}}(i))+x(\delta^{\mbox{\scriptsize in}}(i)) \leq 2$. Via the
identification $y(\delta(i)) \cong x(\delta^{\mbox{\scriptsize out}}(i)) + x(\delta^{\mbox{\scriptsize in}}(i))$ we
see that $y(\delta(i)) \leq 2$ defines a facet of $\uCCCPn$ if
$x(\delta^{\mbox{\scriptsize out}}(i)) \leq 1$ defines a facet of $\CCCPn$. 

Necessity can be shown with similar arguments as 
for the directed counterparts of these inequalities. 
\end{proof}

The inequalities mentioned in Corollary \ref{CfacetsuCCCP} (a)-(c),
(e)-(g) together with the integrality constraints $y_e \in \{0,1\}$
for $e \in E$ provide a characterization of the integer points of
$\uCCCPn$. In this context note that if $|N\setminus S|=2$, the
inequalities in (e) are
equivalent to the well-known parity constraints 
$$y(\delta(j) \setminus
 \{e\}) - y_e \geq 0 \hspace{2cm} (j \in N, e \in \delta(j)) $$ 
mentioned for example in \cite{Bauer}.

The odd cycle exclusion constraints as well as the modified
cardinality forcing inequalities from Corollary \ref{CfacetsCCCP} are
not symmetric nor equivalent to symmetric inequalities. Hence, we did
not derive counterparts of these inequalities for $\uCCCPn$. Of
course, given a valid inequality $cx \leq c_0$ for $\CCCPn$, one
obtains a valid inequality $\tilde{c} y \leq 2c_0$ for $\uCCCPn$ by
setting $\tilde{c}_{ij} := c_{ij}+c_{ji}$ for $i<j$. However, it 
turns out that the counterparts of these two classes of inequalities are irrelevant for a linear description of $\uCCCPn$. 

\subsection{Facets of the undirected cardinality constrained path
  polytope} 

The undirected cardinality constrained $(0,n)$-path polytope $\uCCPPon$
is the symmetric counterpart of $\CCPPon$. Here, $K_{n+1}=(N,E)$
denotes the complete graph on node set $N=\{0,\dots,n\}$.
In the sequel we confine ourselves to the set $\CS$ of cardinality
sequences $c=(c_1,\dots,c_m)$ with $m \geq 2$, $c_1 \geq 2$, and $c
\neq (2,3)$.

\begin{Theorem}
Let $K_{n+1} =(N,E)$ be the complete graph on node set
$N=\{0,\dots,n\}$, $n \geq 4$, and let $c=(c_1,\dots,c_m) \in \CS$ be a
cardinality sequence. Then the
following holds:
\begin{enumerate}
\item[(i)] $\dim \uCCPPon=|E|-3$.
\item[(ii)] The  nonnegativity constraint $y_e \geq
  0$ defines a facet of $\uCCPPon$ if and only if $c \neq (2,n)$ or
  $c=(2,n)$ and $e$ is an internal edge. 
\end{enumerate}
\end{Theorem}

\begin{proof}
(i) All points $y \in \uCCPPon$ satisfy the equations
\begin{eqnarray}
y_{0n} & = & 0, \label{101} \\
y(\delta(0)) & = & 1, \label{102}  \\
y(\delta(n)) & = & 1. \label{103}
\end{eqnarray}
Thus, the dimension of $\uCCPPon$ is at most $|E|-3$. 
When $4 \leq c_i <n$ for some $i \in \{1,\dots,m\}$, then the
statement is implied by Theorem 4.7 of  \cite{Stephan}, saying
that $\dim P_{0,n-\mbox{\scriptsize path}}^{(c_i)}(K_{n+1})=|E|-4$, and the
fact that $m \geq 2$. When $c \in \{(2,n), (3,n), (2,3,n)$,
see~\cite{KS}.  

(ii) When $4 \leq c_i < n$ for some $i \in \{1,\dots,m\}$, then the claim
follows from Theorem 4.9 of  \cite{Stephan} and the fact that
$m \geq 2$. Otherwise, $c=(2,n)$, $c=(3,n)$, or
$c=(2,3,n)$. Then see~\cite{KS}. 
\end{proof}

The concept of symmetric inequalities can be used to derive facet
defining inequalities for $\uCCPPon$ from those for $\CCPPon$. A valid
inequality $cx \leq c_0$ 
for the directed path polytope $\CCPPon$ is said to be \emph{\symmetric}
 if $c_{ij}=c_{ji}$ for all $1 \leq i < j \leq n-1$. It is equivalent to a \symmetric inequality if and only if the system $t_i-t_j=c_{ij}-c_{ji}$ for $1 \leq i <j \leq n-1$ is consistent.
In \cite{Stephan} it was shown that the undirected counterpart
$\bar{c}y \leq c_0$ of a \symmetric inequality $cx
\leq c_0$ (obtained by setting
$\bar{c}_{0i}=c_{0i},\bar{c}_{in}=c_{in}$ for all internal nodes $i$
and $\bar{c}_{ij}=c_{ij}=c_{ij}$ for all $1 \leq i < j \leq n-1$) is
facet defining for $P_{0,n-\mbox{\scriptsize path}}^{(p)}(K_{n+1})$ if
$cx \leq c_0$ is facet defining for $P_{0,n-\mbox{\scriptsize
    path}}^{(p)}(\tdn)$. The same holds for $\CCPPon$ and $\uCCPPon$.

\begin{Corollary} \label{PfacetsuCCCP}
Let $K_{n+1}=(N,E)$ be the complete graph on node set
$N=\{0,\dots,n\}$ with $n \geq 4$, and let $c=(c_1,\dots,c_m) \in \CS$
be a cardinality sequence. Then we have:
\\[0.5em]
(a) The degree constraint $y(\delta(i)) \leq 2$ defines a facet
  of $\uCCPPon$ for every node $i \in N \setminus \{0,n\}$ unless $c=(2,n)$. \\[0.5em]
(b) Let $S$ be a subset of $N$ with $0,n \in S$ and $|S| \leq
  c_1$. Then, the min-cut inequality
\begin{equation}
y((S:N \setminus S)) \geq 2
\end{equation}
induces a facet of $\uCCPPon$ if and only if $|S| \geq 3$ and $|V
\setminus S| \geq 2$. \\[0.5em]
(c) Let $S \subset N$ with $0,n \in S$, $j \in N \setminus S$, and $|S| \geq c_1+1$. Then, the one-sided min-cut inequality 
\begin{equation}
y((S:N \setminus S))- y(\delta(j))\geq 0
\end{equation}
is valid for $\uCCPPon$ and induces a facet of $\uCCPPon$ if and only if
$|N \setminus S| \geq 2$. \\[0.5em]
(d) The cardinality bound $y(E) \geq c_1$ defines a facet of
  $\uCCPPon$ if and only if $c_1 \geq 4$. The cardinality bound $y(E)
  \leq c_m$ defines a facet of $\uCCPPon$ if and only if $c_m < n$.\\[0.5em]
(e) Let $W$ be a subset of $N$ with $0,n \in W$ and $c_p < |W|-1 <
  c_{p+1}$ for some $p \in \{1,\dots,m-1\}$. The cardinality-forcing
  inequality
\begin{equation}
(c_{p+1}-|W|+1) \sum_{i \in W} y(\delta(i)) - (|W|-c_p-1)\sum_{i \in N
  \setminus W} y(\delta(i)) \leq 2 c_p( c_{p+1}-|W|+1) 
\end{equation}
defines a facet of $\uCCPPon$ if and only if $c_{p+1}-|W|+1 \geq 2$ and
$c_{p+1} < n$ or $c_{p+1}=n$ and $|W| = n$. \\[0.5em]
(f) Let $W$ be a subset of $N$ such that $0,n \in W$ and $c_p <
  |W|-1 < c_{p+1}$ for some $p \in \{1,\dots,m-1\}$. The
  cardinality-subgraph inequality
\begin{equation}
    2y(E(W)) - (|W|-c_p-2) y((W:N \setminus W))
    \leq 2 c_p
\end{equation}
 is valid for $\uCCPPon$ and induces a facet of $\uCCPPon$
if and only if $p+1 < m$ or $c_{p+1}=n = |W|$. \\[0.5em]
(g) Let $c=(c_1,\dots,c_m)$ be a cardinality sequence with $m \geq 2$,
$c_1 \geq 2$, and
$c_p$ even for $1 \leq p \leq m$, and let $N = S \; \dot{\cup} \;  T$ be a partition of $N$
with $0 \in S$, $n \in T$. The odd path exclusion constraint
\begin{equation}
y(E(S))+ y(E(T)) \geq 1
\end{equation}
is valid for $\uCCPPon$ and defines a facet of  $\uCCPPon$ if and only
if (i) $c_1=2$ and $|S|, |T| \geq \frac{c_2}{2}+1$, or (ii) $c_1 \geq 4$ and 
$|S|,|T| \geq \frac{c_2}{2}$. \\[0.5em]
(h) Let $c=(c_1,\dots,c_m)$ be a cardinality sequence with $m \geq 2$,
$c_1 \geq 3$, and
$c_p$ odd for $1 \leq p \leq m$, and let $N = S \; \dot{\cup} \;  T$ be a partition of $N$
with $0,n \in S$. The even path exclusion constraint
\begin{equation}
y(E(S))+ y(E(T)) \geq 1
\end{equation}
is valid for $\uCCPPon$ and defines a facet of  $\uCCPPon$ if and only
if ($\alpha$) $c_1=3$, $|S|-1 \geq \frac{c_2+1}{2}$, and $|T| \geq
\frac{c_2-1}{2}$, or ($\beta$) $c_1 \geq 5$ and $\min (|S|-1,|T|) \geq
\frac{c_2-1}{2}$.\hfill $\Box$
\end{Corollary}
As already mentioned, the modified cardinality forcing inequalities
\eqref{MCF} are not equivalent to \symmetric inequalities.

\newpage

\section{Concluding remarks}
Restricting the set of feasible solutions of a combinatorial
optimization problem to those that satisfy some specified cardinality
constraints always  can be done by adding the corresponding
cardinality forcing inequalities inherited from the polytope
associated with the respective cardinality homogeneous set system. However, as we have demonstrated at the example of paths and cycles, one may end with rather weak formulations unless this is done carefully: Imposing the restrictions on the number of vertices leads to formulations with facet defining inequalities, while the straight-forward approach using the arcs does not result in strong inequalities.

It would be interesting to see whether this is similar for cardinality
restricted versions of other optimization problems. Moreover, we believe
that there should be other interesting situations where knowledge on a
master polyhedron (like the cardinality homogeneous set systems polyhedron) and on a polyhedron associated with particular combinatorial structures (like paths and cycles) can be brought into fruitful interplay.

\newpage
\noindent
\small{E-mail address: \texttt{kaibel@ovgu.de}}\\
\small{E-mail address: \texttt{stephan@math.tu-berlin.de}}


\begin{thebibliography}{99}
\bibitem[1]{BO} E. Balas and M. Oosten, \emph{On the cycle polytope of
    a directed graph}, Networks 36 No. 1 (2000), pp. 34-46.
\bibitem[2]{BST} E. Balas and R. Stephan, \emph{On the cycle polytope of
    a directed graph and its relaxations}, submitted to Networks.
\bibitem[3]{Bauer} P. Bauer, \emph{A Polyhedral Approach to the
    Weighted Girth Problem}, Aachen 1995.
\bibitem[4]{BLS} P. Bauer, J.T. Linderoth, and M.W.P. Savelsbergh,
  \emph{A branch and cut approach to the cardinality constrained
    circuit problem}, Mathematical Programming, Ser. A 91 (2002),
  pp. 307-348.
  \bibitem[5]{BHHS}  E. Boros, P. Hammer, M. Hartmann, and R. Shamir, \emph{Balancing problems in acyclic networks}, Discrete Applied Mathematics 49 (1994),
  pp. 77-93.
\bibitem[6]{CP} C. Coullard and W.R. Pulleyblank, \emph{On cycle cones and polyhedra}, Linear Algebra Appl. 114/115 (1989), pp. 613-640.
\bibitem[7]{DG} G. Dahl and L. Gouveia, \emph{On the directed
    hop-constrained shortest path problem}, Operations Research
  Letters 32 (2004), pp. 15-22.
\bibitem[8]{DR} G. Dahl and B. Realfsen, \emph{The
    Cardinality-Constrained Shortest Path Problem in 2-Graphs},
  Networks 36 No. 1 (2000), pp. 1-8.
\bibitem[9]{FS}  U. Feige and M. Seltser, \emph{On the densest k-subgraph problem}, Technical report,
Department of Applied Mathematics and Computer Science, The Weizmann
Institute, Rehobot, 1997.
\bibitem[10]{Fischetti} M. Fischetti, \emph{Clique tree inequalites
    define facets of the asymmetric traveling salesman polytope},
  Discrete Applied Mathematics 56 (1995), pp. 9-18. 
\bibitem[11]{polymake} E. Gawrilow and M. Joswig, \emph{polymake: A
    framework for analyzing convex polytopes}. In: G. Kalai and G.M. Ziegler (eds.): Polytopes ~ Combinatorics and Computation
(DMV-Seminars, pp. 43~74) Basel: Birkh¨auser-Verlag Basel 2000,
see also http://www.math.tu-berlin.de/polymake
\bibitem[12]{Groetschel} M. Gr\"otschel, \emph{Cardinality homogeneous
    set systems, cycles in matroids, and associated polytopes}, in:
  M. Gr\"otschel, \emph{The sharpest cut. The impact of Manfred Padberg
  and his work}, MPS-SIAM Series on Optimization 4, SIAM, 2004,
pp. 199-216.
\bibitem[13]{GP} M. Gr\"otschel and M.W. Padberg, \textit{Polyhedral theory}, in: E.L. Lawler et al (eds.),
\textit{The traveling salesman problem. A guided tour of combinatorial optimization}, Chichester, New York, and others, 1985, pp. 
251-305.
\bibitem[14]{HO} M. Hartmann and \"O. \"Ozl\"uk, \emph{Facets of the
    $p$-cycle polytope}, Discrete Applied Mathematics 112 (2001),
  pp. 147-178.
\bibitem[15]{KS} V.~Kaibel and R.~Stephan, \emph{On cardinality
    constrained cycle and path polytopes}, ZIB-Report, October, 2007,
  available 
  at \texttt{www.zib.de/bib/pub/index.en.html} \,.
\bibitem[16]{KMV} M. Kovalev, J.-F. Maurras, and Y. Vax\'{e}s,
  \emph{On the convex hull of 3-cycles of the complete graph}, Pesquisa Operational, 23 (2003) 99-109.
\bibitem[17]{MN1} J.-F. Maurras and V.H. Nguyen, \emph{On the linear description of the 3-cycle polytope}, European Journal of Operational Research, 1998.
\bibitem[18]{MN2} J.-F. Maurras and V.H. Nguyen, \emph{On the linear
    description of the k-cycle polytope, $PC_n^k$}, International
  Transactions in Operational Research 8 (2001), pp. 673-692.
\bibitem[19]{NW}G.L. Nemhauser and L.A. Wolsey, Integer and Combinatorial Optimization, Wiley: New York, 1988.
\bibitem[20]{Nguyen} V.H. Nguyen,\emph{A complete description for the k-path polyhedron}, Proceedings of the Fifth International Conference on Modelling, Computation and Optimization in Information Systems and Management Science, pp. 249-255. 
\bibitem[21]{Stephan} R. Stephan, \emph{Facets of the (s,t)-p-path polytope},
  arXiv: math.OC/0606308, submitted, June 2006. 
\bibitem[22]{StephanDipl} R. Stephan, \emph{Polytopes associated with
    length restricted directed circuits}, Master Thesis, Technische
  Universit\"at Berlin, 2005.
\bibitem[23]{Schrijver2003} A. Schrijver, \emph{Combinatorial
    Optimization}, Vol. A, Berlin et al., 2003.
\end{thebibliography}
\end{document}